\documentclass[reqno]{amsart}

\usepackage{etex}
\usepackage{amsfonts}
\usepackage{amscd}
\usepackage{amsbsy}
\usepackage{mathtools}
\usepackage{amsxtra}
\usepackage{amssymb}
\usepackage{epsfig}
\usepackage{epic,eepic}
\usepackage{graphicx}
\usepackage[all]{xy}
\usepackage{psfrag}
\usepackage{amsmath}
\usepackage{amsthm}
\usepackage{setspace}
\usepackage{hyperref}
\usepackage{url}
\usepackage{here}
\usepackage{todonotes}
\newlength{\halfbls}

\usepackage{xparse}
\usepackage{xspace}
\usepackage[normalem]{ulem}

\usepackage{booktabs}
\usepackage{environ}
\usepackage{blindtext}
\usepackage{caption}

\usepackage{tikz-cd}

\usepackage{tikz}
\usetikzlibrary{patterns}
\usetikzlibrary{calc}
\usetikzlibrary{decorations.pathreplacing,decorations.markings}
\usetikzlibrary{arrows}
\usetikzlibrary{decorations.markings, decorations.shapes}
\usetikzlibrary{positioning}
\usetikzlibrary{intersections}
\usetikzlibrary{decorations.pathmorphing}
\usetikzlibrary{arrows.meta,bending}

\usepackage{thmtools}

\DeclareRobustCommand{\SkipTocEntry}[9]{}
\DeclareMathOperator*\geht\longrightarrow

\DeclareMathAlphabet\mathbfcal{OMS}{cmsy}{b}{n}
\usepackage{adjustbox}

\theoremstyle{plain}
\newtheorem{Defi}{Definition}[section]  
   \newtheorem{Prop}[Defi]{Proposition}
\newtheorem{Lemma}[Defi]{Lemma}    \newtheorem{Cor}[Defi]{Corollary}
\newtheorem{Thm}[Defi]{Theorem}

\theoremstyle{remark}
\newtheorem{Ex}[Defi]{Example}

\newtheorem{Rem}[Defi]{Remark}

\newcommand{\CC}{\mathbb{C}}

\newcommand{\NN}{\mathbb{N}}
 
\newcommand{\PP}{\mathbb{P}}

\newcommand{\ZZ}{\mathbb{Z}}
\def\Z{\ZZ} 

\newcommand{\cMM}{{\mathcal M}}

  \def\M{\cMM}

\renewcommand\P{{\bf P}}

\newcommand{\Pic}{{\rm Pic}}

\def\={\;=\;}  \def\+{\,+\,}

\DeclareMathAlphabet{\eucal}{U}{eus}{m}{n}
\DeclareMathAlphabet{\newcal}{U}{dutchcal}{m}{n}

\def\Z{\ZZ}

\def\be{\begin{equation}}   \def\ee{\end{equation}}     \def\bes{\begin{equation*}}    \def\ees{\end{equation*}}
\def\ba{\be\begin{aligned}} \def\ea{\end{aligned}\ee}   \def\bas{\bes\begin{aligned}}  \def\eas{\end{aligned}\ees}

\newcommand{\moduli}[1][g]{{\mathcal M}_{#1}}

\def\M{\cMM}

\newcounter{savedtocdepth}
\newcommand*{\SaveTocDepth}[1]{%
  \addtocontents{toc}{%
    \protect\setcounter{savedtocdepth}{\protect\value{tocdepth}}%
    \protect\setcounter{tocdepth}{#1}%
  }%
}

\DeclareDocumentCommand{\LMS}{ O{\mu} O{g,n}} {\Xi\overline{\mathcal{M}}_{#2}(#1)}
\DeclareDocumentCommand{\kLMS}{ O{\mu} O{g,n}} {\Xi^k\!\overline{\mathcal{M}}_{#2}(#1)}

\newcommand{\hslashslash}{%
\lapbox[\width]{-0.15em}{\raisebox{.05ex}{%
    \scalebox{.7}{%
      \rotatebox[origin=c]{22}{$-$}%
    }%
  }%
}}

\newcommand{\bslash}{%
  {%
   \vphantom{d}%
   \ooalign{\kern.05em\smash{\hslashslash}\hidewidth\cr$b$\cr}%
   \kern.05em
  }%
}

\usepackage[style=alphabetic,
            isbn=false,
            doi=false,
            url=false,
            maxnames=5
           ]{biblatex}
\defbibheading{bibliography}[\bibname]{%
  \section*{References}%
}

\bibliography{literatur}

\usepackage{breqn}
\allowdisplaybreaks[2]

\title[On fibrations approaching the Arakelov equality]
      {On fibrations approaching the Arakelov equality}

\begin{document}
\author{Maximilian Bieri}
\email{mbieri@math.uni-frankfurt.de}

\thanks{Research is partially supported  
  by the LOEWE-Schwerpunkt
``Uniformisierte Strukturen in Arithmetik und Geometrie'' and by the Friedrich-Ebert-Stiftung.}
\address{
Institut f\"ur Mathematik, Goethe--Universit\"at Frankfurt,
Robert-Mayer-Str. 6--8,
60325 Frankfurt am Main, Germany
}

\begin{abstract}
The sum of Lyapunov exponents~$L_f$ of a semi-stable fibration is the ratio of the degree of the Hodge bundle by the Euler characteristic of the base. This ratio is bounded from above by the Arakelov inequality. Sheng-Li Tan showed that for fiber genus~$g\geq 2$ the Arakelov equality is never attained. We investigate whether there are sequences of fibrations approaching asymptotically the Arakelov bound. The answer turns out to be no, if the fibration is smooth, or non-hyperelliptic, or has a small base genus.
\par
Moreover, we construct examples of semi-stable fibrations showing that Teichmüller curves are not attaining the maximal possible value of~$L_f$.
\end{abstract}
\maketitle
\tableofcontents
\SaveTocDepth{1}

\section{Introduction}
\label{sec:intro}

The closure of a curve in~$\moduli$ can equivalently be regarded as a {\em semi-stable fibration~$f:X\rightarrow C$ of genus~$g$}. These fibrations come with an additional object, the relative dualizing sheaf~$\omega_{X/C}$. The sheaf defines two invariants, the self-intersection number and the degree of its push-forward. A well known ratio among these numbers is the {\em slope~$\lambda_f$} defined as the quotient of these two invariants, i.e.\ as
$$\lambda_f:=\frac{\omega_{X/C}^2}{\deg f_*\omega_{X/C}}.$$
Its definition is motivated by trying to understand effective divisors on~$\mathcal{M}_{g}$, which goes back to Cornalba and Harris~\cite{cornalbaharris}. Since curves are dual to divisors, the slope of nef curves in~$\mathcal{M}_g$ can be used to estimate the slope of effective divisors. This problem for divisors, known as slope conjecture, is still open and understanding the slopes of fibrations is an important problem in surface geometry. The slope of fibrations is bounded by the \emph{slope inequality}
\begin{equation}
\label{eqslopeineguality}
\frac{4(g-1)}{g}\leq\lambda_f\leq 12,
\end{equation}
where the upper bound is attained precisely by smooth fibrations and the lower bound can only be attained by hyperelliptic fibrations with special properties, to be discussed in Section~\ref{sechypfib}.
\par
In the present paper we propose to study a new invariant for semi-stable fibrations, namely the {\em sum of non-negative Lyapunov exponents~$L_f$}. Lyapunov exponents originate from dynamical systems and have been brought to the study of the geometry of the moduli space of curves through the connection with billiards and the ${\rm SL}_2(\mathbb{R})$-action on the moduli space of flat surfaces. They measure the growth rate of cohomology classes of a flat bundle when parallel transported along a geodesic flow. Since the growth rate definition of Lyapunov exponents can equivalently be phrased as the speed of the image of geodesics under the period mapping, it is tempting to nickname the invariant~$L_f$ we are interested in as the {\em speed} of the semi-stable fibration - and we will succumb the temptation from now on. While the nature of individual Lyapunov exponents is rather unclear, the sum of the non-negative exponents, or the speed, can be expressed by a ratio of classical invariants, namely as
$$L_f = 2\cdot\frac{\deg f_*\omega_{X/C}}{2g_C-2+s},$$
where~$s$ is the number of singular fibers of~$f$. The {\em Arakelov inequality for semi-stable fibrations}
$$L_f\leq g$$ gives an upper bound of this ratio. Since the main result of~\cite{tanminnumber} shows that the Arakelov equality can never be attained for a fibration with fiber genus~$g \geq 2$ the question arise, whether there are sequences of semi-stable fibrations reaching asymptotically the Arakelov equality.
\par
We are going to partially answer this question by proving numerical bounds that are strengthening the Arakelov inequality for the speed of various classes of semi-stable fibrations. Initially, and not surprisingly, a smooth fibration, a so-called \emph{Kodaira fibration}, is far off from reaching the Arakelov bound.
\begin{restatable}{Prop}{propsurfacebundle}
\label{surfacebundle}
	Let~$f:X\to C$ be a Kodaira fibration of genus~$g$. Then
	$$L_f\leq\frac{g-1}{3}.$$
\end{restatable}
Therefore we focus subsequently on semi-stable fibrations~$f:X\rightarrow C$ which have at least one singular fiber. The next result asserts that a low base genus of~$C$ yields a better bound than the Arakelov inequality for the speed.
\begin{restatable}{Thm}{lowbasegen}\label{thmlowbasegen}
Let $f: X \to C$ be a semi-stable fibration of genus~$g\geq 2$ with~$s$ singular fibers and let~$m\in\mathbb{N}$ be a number such that
$$\frac{s}{2g_C-2+s}\geq\frac{1}{m}.$$
Then the speed of~$f$ is bounded from above by
$$L_f\leq g\cdot\left(1-\frac{1}{18m}\right).$$
\end{restatable}
For the sake of illustration we exhibit some values up to three digits of upper bounds given by Theorem~\ref{thmlowbasegen} for small genera in Table~\ref{Tabelle1}.
\begin{table}[htbp]
\centering
\begin{tabular}{lccccccc}
\toprule  & $g=2$ & $g=3$ & $g=4$ & $g=5$ & $g=6$ & $g=7$ & $g=8$ \\
\midrule $g_C\leq 1$ & $1.889$ & $2.833$ & $3.778$ & $4.722$ & $5.667$ & $6.611$ & $7.556$\\
 $g_C=2$ & $1.944$ & $2.917$ & $3.889$ & $4.861$ & $5.833$ & $6.806$ & $7.778$\\
\bottomrule
\end{tabular}
\caption{\label{Tabelle1}Upper bound of the speed for small base genus.}
\end{table}

Slope and speed of semi-stable fibrations are related. In fact the combination of the strict canonical class inequality
$$\omega_{X/C}^2<(2g-2)\cdot(2g_C-2+s)$$
and the slope inequality~\eqref{eqslopeineguality} show that high slope implies low speed. Equality in the lower slope bound can only be attained by hyperelliptic fibrations with special properties, as we have mentioned above. There are various stricter, lower bounds of the slope known for various classes of fibrations. In particular we will use lower bounds for non-hyperelliptic fibrations of genus~$g=3,4,5$ by~\cite{konno}. For~$g\geq 6$, we improve a result of~\cite{gonality} to get lower bounds for the slope of non-hyperelliptic fibrations of genus~$g$, which implies strict upper bounds for the speed.
\begin{restatable}{Thm}{speednonhyp}\label{thmspeednonhyp}
Let~$f:X\to C$ be a semi-stable, non-hyperelliptic fibration of genus~$g$. Then we have
\begin{itemize}
\item[i)] $L_f<\frac{8}{3}$ for~$g=3$,
\item[ii)] $L_f<\frac{7}{2}$ for~$g=4$,
\item[iii)] $L_f< 4$ for~$g=5$,
\item[iv)] $L_f<\frac{8g+4}{9}$ for~$6\leq g\leq 12$,
\item[v)] $L_f< g-1$ for~$g\geq 13$.
\end{itemize}
\end{restatable}
Again we illustrate the approximate values to sketch a rough overview in Table~\ref{Tabelle2}.
\begin{table}[htbp]
\centering
\begin{tabular}{ccccccccc}
\toprule   $g=3$ & $g=4$ & $g=5$ & $g=6$ & $g=7$ & $g=8$ & $g=9$ & $g=10$ & $g=11$  \\
\midrule  $2.667$ & $3.5$ & $4$ & $5.778$ & $6.667$ & $7.556$ & $8.444$ & $9.333$ & $10.222$\\
\bottomrule
\end{tabular}
\caption{\label{Tabelle2}Upper bound of speed for non-hyperelliptic fibrations.}
\end{table}
By combining Proposition~\ref{surfacebundle}, Theorem~\ref{thmlowbasegen} and Theorem~\ref{thmspeednonhyp} we see that a candidate sequence reaching asymptotically the Arakelov equality must contain hyperelliptic, non-smooth fibrations, whose base genus~$g_C$ are asymptotically approaching infinity. Therefore the focus will be on the construction of hyperelliptic fibrations over smooth curves~$C$ with arbitrary base genus. Prior to this paper, the curves with the highest speed known were Teich\-mül\-ler curves, with a speed of $L_f = (g+1)/2$ as stated in~\cite{nonvaryingsums}. We will give examples of hyperelliptic fibrations showing that Teichmüller curves are not attaining the maximal possible speed. 
\begin{restatable}{Thm}{thmexistenz}\label{maintheorem1}
	Let~$g\geq 2$ be a number. Then there are semi-stable, hyperelliptic fibrations~$f:X\rightarrow C$ of genus~$g$ with a speed of
	$$L_f>\frac{g+1}{2}.$$
	In particular there are semi-stable, hyperelliptic fibrations~$f$ with
	\begin{itemize}
	\item[i)] $L_f=\frac{8}{5}$ for~$g=2$,
	\item[ii)] $L_f=\frac{8}{3}$ for~$g=3$,
	\item[iii)] $L_f=g-\lfloor\frac{g+1}{4}\rfloor$ for~$g\geq 4$ odd,
	\item[iv)] $L_f=g-\frac{g^2-2g}{2g+2}$ for~$g\geq 4$ even,
	\item[v)] $L_f=g-\frac{g}{4}$ for~$g\equiv 0\mod 4$.
	\end{itemize}
\end{restatable}
Again we exhibit some values for small fiber genus in Table~\ref{Tabelle3}.
\begin{table}[htbp]
\centering
\begin{tabular}{cccccccc}
\toprule    $g=2$ & $g=3$ & $g=4$ & $g=5$ & $g=6$ & $g=7$ & $g=8$ & $g=9$  \\
\midrule  $1.6$ & $2.667$ & $3.2$ & $4$ & $4.286$ & $5$ & $6$ & $7$  \\
\bottomrule
\end{tabular}
\caption{\label{Tabelle3}High speed examples of semi-stable fibrations.}
\end{table}

\subsection*{Acknowledgements}

I am very grateful to Martin Möller, the advisor of my PhD thesis, for his constant support and many inspiring discussions.
 

\section{Sum of Lyapunov exponents}
\label{sec:lyapunov}

This section contains the motivation of the nickname 'speed' as well as some older results on the speed of semi-stable fibrations coming from Teichmüller curves.
\par
Lyapunov exponents of a semi-stable fibration measure the asymptotic growth rate of cohomology classes under parallel transport along the geodesic flow~$g_t$ of the base. More precisely, given a semi-stable fibration of genus~$g\geq 2$, we denote by~$\Delta$ the boundary points of~$C$. Then we start from the observation that the base~$C -\Delta$ of~$f$ is hyperbolic, i.e.\ uniformized by the upper half plane~$\mathbb{H}$. Consequently, we can look at the geodesic flow $g_t$ on $\mathbb{H}$, or rather on the unit tangent bundle~$T^{1} \mathbb{H}$ and its image on $T^1C$. We consider the bundle $V$ on $T^1C$ which has as fibers over~$c\in C$ the cohomology $H^1(F_c,\mathbb{R})$. The Hodge norm makes this bundle a normed vector bundle, but the norm is anything but constant under parallel transport by~$g_t$. Rather, in this situation an ergodicity hypothesis is satisfied and consequently, Oseledets theorem~\cite{oseledets} provides a filtration~$V=V_{1}\supsetneq\ldots\supsetneq V_{k}\supset 0$ and real numbers~$\widetilde{\lambda}_{1},\ldots,\widetilde{\lambda}_{k}$, such that for almost all~$c\in T^1C$ and~$v\in (V_{i})_{c}\setminus\{0\}$, one has
\begin{equation*}
{\lVert}g_{t}(v){\rVert}=\exp(\widetilde{\lambda}_{i}t+o(t)).
\end{equation*} 
The {\em Lyapunov exponents} 
$$ \lambda_1 \geq \cdots \geq \lambda_g \geq 0 \geq \lambda_{g+1}\geq \cdots \geq \lambda_{2g}$$
are the sequence of the normalization of the exponents $\widetilde{\lambda}_{i}$ repeated with multiplicity equal to the rank of~$V_{i}/V_{i+1}$ in order to get~$2g$ exponents. The symplectic form on the bundle $V$ implies that the value of the Lyapunov exponents are symmetric with respect to the origin, $\lambda_{2g+1-i} = -\lambda_i$ and by the normalization of the Lyapunov exponents we always have the bound $|\lambda_i| \leq 1$. For more information we refer to~\cite[Section~9]{bouwmoeller}.
\par
Little is known about individual Lyapunov exponents. The sum $L_f:=\sum_{i=1}^g \lambda_i$ of non-negative Lyapunov exponents is related to a ratio of the classical relative invariants of fibrations. Kontsevich~\cite{sumlyapunov} showed that 
$$ 2\cdot\frac{\deg f_*\omega_{X/C}}{|\chi(C)|} \,=\, \sum_{i=1}^g \lambda_i=L_f,$$
where~$f_*\omega_{X/C}$ is the Hodge bundle of the semi-stable fibration~$f$.
\par
We take advantage of an alternative intuitive explanation for the quantity measured by Lyapunov exponents. Associated with any fibration~$f: X \rightarrow C$ there is the period map~$p:\mathbb{H} \to \mathbb{H}_g$ from the universal cover of the curve~$C$ to the Siegel upper half space. Lyapunov exponents measure the speed the image curves under~$p$ travel when traveling with unit speed along the geodesic flow in~$\mathbb{H}$, in an average over all geodesics in~$\mathbb{H}$. This justifies to call~$L_f$ the {\em speed} of the fibration instead of using the bulky expression 'sum of all non-negative Lyapunov exponents'.

\subsection{Results for Teich\-mül\-ler curves}
\label{speedTeichm}
The invariant 'sum of non-negative Lyapunov exponents' was originally introduced to Teich\-mül\-ler curves (see~\cite{zorichlyapunovexponents}) and computed. For Teich\-mül\-ler curves constructed from cyclic covers, the speed has been computed in~\cite{bouwmoeller} and~\cite{weierstrassfiltration} contains some conjectural estimates for further examples.
\par If~$f$ is a fibration coming from a Teich\-mül\-ler curve generated by a flat surface~$(X,\omega)$, then the slope~$\lambda_f$ and the speed~$L_f$ are related by
$$ \lambda_f \,=\frac{12\kappa_f}{L_f},$$
where $\kappa_f$ is a constant that only depends on the order of zeros of~$\omega$ (see~{\cite[Proposition 4.5]{nonvaryingsums}}). For Teich\-mül\-ler curves a lot is known about~$\lambda_f$ and hence equivalently about~$L_f$.
\par
For each fixed type of zeros of the flat surface there is a formula for the limit of~$L_f$ as $\chi(C)$ grows, the Lyapunov exponents for strata (see~\cite{sumlyapunov}). For~$g=2$ and for several other types of zeros of the flat surface, the invariant~$L_f$ is the same for all Teich\-mül\-ler curves living in the same stratum. This was established in~\cite{nonvaryingsums}, where this phenomenon was called \emph{non-varying}.
\begin{Thm}[{\cite[Corollary 3.5]{nonvaryingsums}}]
\label{theoremnonvarying}
Let~$f:X\rightarrow C$ be a hyperelliptic fibration of genus~$g$ coming from a Teich\-mül\-ler curve generated by a flat surface~$(X,\omega)$. Then we have
$$L_f=\frac{g^2}{2g-1} \text{ and } \lambda_f=\frac{4(g-1)}{g} \text{ if } (X,\omega)\in\Omega\mathcal{M}_g^{hyp}(2g-2),$$
$$L_f=\frac{g+1}{2} \text{ and } \lambda_f=\frac{4(g-1)}{g} \text{ if } (X,\omega)\in\Omega\mathcal{M}_g^{hyp}(g-1,g-1).$$
\end{Thm}
The speed~$(g+1)/2$ for semi-stable fibrations is in fact the highest possible for any Teich\-mül\-ler curve by a result of~\cite{UpperboundsTeichm}.
\begin{Thm}[{\cite[Theorem 1.2]{UpperboundsTeichm}}]
\label{upperboundteich}
Let~$f:X\rightarrow C$ be a fibration of genus~$g$ coming from a Teich\-mül\-ler curve generated by a flat surface~$(X,\omega)$. Then we have
$$L_f\leq\frac{g+1}{2}.$$
Furthermore we have equality if and only if~$(X,\omega)\in\Omega\mathcal{M}_g^{hyp}(g-1,g-1)$ or $(X,\omega)\in\Omega\mathcal{M}_g^{hyp}(1^{2g-2})$.
\end{Thm}
To our knowledge, for any~$g\geq 2$ there are no previously known semi-stable fibrations with a speed greater than~$(g+1)/2$. For hyperelliptic Teich\-mül\-ler curves, in particular when the flat surface has only one zero, a lower bound for~$L_f$ is known by~\cite{weierstrassfiltration}, which is asymptotically optimal as~$g \to \infty$. But even for Teich\-mül\-ler curves beyond the cases where the type of zeros fall into the non-varying phenomenon, there are intriguing open problems about the speed. Already for fibrations with fiber genus~$3$ and when the flat surface has four simple zeros it is challenging to get good lower bounds for $L_f$. Any statement like this must except the unique (as shown in \cite{shimundteichm}) Teich\-mül\-ler curve of this type, where the speed is as low as it could possibly be, i.e.\ $L_f=1$. 

\section{Fibrations}
\label{fibrations}

Let~$X$ be a smooth projective surface and let~$C$ be a smooth curve. We will call $f: X \rightarrow C$ a \emph{fibration}, if the map~$f$ is surjective and all fibers are connected curves. If all smooth fibers are isomorphic to a fixed curve, we speak of a \emph{isotrivial fibration}, if all fibers are smooth, the fibration is \emph{smooth} and if the fibration is isotrivial and smooth it is called \emph{locally trivial}. If no fiber contains a~$(-1)$-curve we say the fibration is \emph{relatively minimal}.
\par
Let~$g_C:=g(C)$ be the genus of the base curve~$C$, respectively $g :=g(F)$ the genus of a general fiber~$F$ of~$f$. If~$g=0$, then the fibration~$f:X\rightarrow C$ is nothing but a \emph{ruled surface}. We are going to introduce them in Section~\ref{chapterruledsurf} in another context. If~$g=1$, then~$f$ is called an \emph{elliptic fibration}. Since elliptic fibrations have been classified in the past (see~\cite{beauvilleelliptic},~\cite{kodairaelliptic2} and~\cite{kodairaelliptic1}) we will from now on always assume a fiber genus of~$g\geq 2$.
\par
The basic invariants of~$f:X\rightarrow C$ arise as the self-intersection number and the degree of the push-forward of the \emph{relative canonical sheaf} 
$$\omega_{f}:=\omega_X\otimes(f^{*}\omega_C)^{\vee}.$$
The push-forward~$f_*\omega_{f}$ is a vector bundle on~$C$ of rank~$g$, which we will call \emph{Hodge bundle of~$f$}. The basic relative invariants for a fibration are now
$$\omega_f^2, \qquad \deg f_*\omega_f \qquad \text{and } \delta_f:=c_2(X)-4(g_C-1)(g-1).$$
The formula of these invariants is
$$\begin{aligned}
 \omega_f^2 = c_1^2(X) - 8(g_C-1)(g-1), \\
 \deg f_* \omega_{f} = \chi(\mathcal{O}_X) - (g_C-1)(g-1).
\end{aligned}$$
Finally we have two basic quantities, since all three invariants are related by Noether's formula
\begin{equation}
\label{Noether}
\omega_{f}^2=12\deg f_* \omega_{f}-\delta_f.
\end{equation}
We will often denote the number~$\deg f_* \omega_{f}$ by~$\chi_f$. We say that the fibration~$f$ is \emph{semi-stable} if all fibers~$F$ are semi-stable, i.e.\ $F$ is reduced, a singularity in~$F$ is at most a node and every smooth rational component~$K$ of~$F$ has at least two points in common with~$F - K$. Note that every semi-stable fibration is relatively minimal and that~$\delta_f$ is nothing but the number of nodes in all singular fibers. In fact we have~$\delta_f=0$ if and only if~$f$ is a fibration that has no singular fibers, i.e.\ a \emph{smooth fibration}. For a relatively minimal fibration it is well-known that~$\omega_f^2$ and~$\chi_f$ are both non-negative numbers and moreover is~$\chi_f$ zero if and only if~$f$ is locally trivial~{\cite[Theorem~III.18.2]{compactcomplex}}. All fibrations will from now on be not isotrivial, unless stated otherwise.

\subsection{Slope of fibrations}
\label{chapterslope}
For a fibration~$f:X\rightarrow X$, we already defined the \emph{slope}~$\lambda_f$, defined as the quotient
$$\lambda_f:=\frac{\omega_f^2}{\chi_f}.$$

The slope is well-defined, since we assume~$f$ to be not isotrivial and therefore we have~$\chi_f>0$. In fact we do not require~$f:X\rightarrow C$ to be semi-stable, we just need the assumption that the fibration is relatively minimal. Due to Noether's formula~\eqref{Noether} we can directly deduce the maximum~$\lambda_f\leq 12$, which is attained if and only if the fibration is smooth. On the other hand Cornalba-Harris~\cite{cornalbaharris} for semi-stable fibrations and Xiao~\cite{lowslope} with different methods for relatively minimal fibrations  have shown the \emph{slope inequality}.
\begin{Thm}[{\cite[Theorem 2]{lowslope}}]
\label{slopeinequality}
Let~$f:X\rightarrow C$ be a relatively minimal fibration of genus~$g\geq 2$. Then we have
$$\frac{4(g-1)}{g}\leq\lambda_f\leq 12.$$
Furthermore we have~$\lambda_f=12$ if and only if~$f$ has no singular fibers.
\end{Thm}
Xiao also conjectured that equality in the lower slope bound~$4(g-1)/g$ can only be attained by hyperelliptic fibrations, which was shown to be true in~{\cite[Proposition 2.6]{konno}}. The lower bound is in fact sharp, we will construct examples in low genus attaining the lower slope inequality later (see Example~\ref{bspbeauville}).

\subsubsection{Slope of non-hyperelliptic fibrations}
We call a fibration~$f:X\rightarrow C$ \emph{non-hyperelliptic}, if the general fiber is a non-hyperelliptic curve of genus~$g$. Note that this implies~$g\geq 3$. We just mentioned that the lower slope inequality can only be attained by hyperelliptic fibrations. Hence there might exist better lower bounds for the slope of a non-hyperelliptic fibration. A lot is known in this direction, for example in~\cite{konnogen6} a better bound for a trigonal fibration was proven and generalised to~$c$-gonal fibrations in~\cite{gonality}. Unfortunately, precise bounds for non-hyperelliptic fibrations are still only known for small fiber genera. Let~$f:X\rightarrow C$ be a relatively minimal, non-hyperelliptic fibration. Then we have as strict lower bounds for the slope

\begin{align}
\label{lowboundsnonhyp}
\lambda_f \geq
\left\{
\begin{array}{l l l l r}
& 3 &  \text{ for }g=3, & \text{ } & \text{\cite{horikawagen3} and \cite{konnogen3}}\\ 
& \frac{24}{7} & \text{ for }g=4, & \text{ }  & \text{\cite{chen} and \cite{konno}} \\
& 4 & \text{ for }g=5, & \text{ }  & \text{\cite{konno}}\\ 
\end{array}
\right.
\end{align}
Additionally it was recently discovered in~\cite{gonality} that the slope of a non-hy\-per\-el\-lip\-tic fibration of genus~$g\geq 16$ has as lower bound~$\lambda_f\geq 4$. We follow here their work to state a lower slope bound for a general non-hyperelliptic fibration.
\par
To address the slope one usually works with the Harder-Narasimhan filtration for the Hodgbe-bundle~$f_*\omega_f$ as introduced by Xiao in~\cite{lowslope}. For this reason we let~$\mathcal{E}$ be a non-zero vector bundle on the curve~$C$. We define the~\emph{slope of~$\mathcal{E}$} as the rational number
\begin{equation*}
\mu(\mathcal{E}):=\frac{\deg\mathcal{E}}{\rm{rank}\mathcal{E}}.
\end{equation*}
We call a bundle~$\mathcal{E}$ \emph{stable}, respectively \emph{semi-stable}, if for every non-zero subbundle~$\mathcal{E}'\subseteq\mathcal{E}$ we have~$\mu(\mathcal{E}')<\mu(\mathcal{E})$, resp.~$\mu(\mathcal{E}')\leq\mu(\mathcal{E})$. We further call a bundle~$\mathcal{E}$ \emph{positive}, respectively \emph{semi-positive}, if for all quotient bundles~$\mathcal{E}\twoheadrightarrow\mathcal{Q}$ we have~$\deg\mathcal{Q}>0$, resp.~$\deg\mathcal{Q}\geq 0$. It is well known that the Hodge bundle~$f_*\omega_f$ for a relatively minimal fibration~$f:X\rightarrow C$ is semi-positive. The \emph{Harder-Narasimhan filtration} for a vector bundle~$\mathcal{E}$ is an unique filtration
$$0=\mathcal{E}_0\subseteq\mathcal{E}_1\subseteq\mathcal{E}_2\dots\subseteq\mathcal{E}_n=\mathcal{E},$$
such that for~$i=1,\dots,n$
\begin{itemize}
\item[•] each of the quotients~$\mathcal{E}_i/\mathcal{E}_{i-1}$ is semi-stable of slope~$\mu_i:=\mu(\mathcal{E}_i/\mathcal{E}_{i-1})$,
\item[•] the slopes are strictly decreasing, that is~$\mu_i>\mu_j$ for~$i<j$.
\end{itemize}
This filtration originates from~\cite{HN-filtration} which explains its name. We will from now on always consider the Harder-Narasimhan filtration of the Hodge bundle~$f_*\omega_f$ and abbreviate it by H-N filtration. It is clear from the semi-positivity of~$f_*\omega_f$, that we have~$\mu_n\geq 0$. By defining~$r_i:=\rm{rank}\mathcal{E}_i$ and setting~$\mu_{n+1}:=0$ we can compute the degree of the Hodge bundle by
\begin{equation}
\label{summechif}
\chi_f=\sum\limits_{i=1}^{n} r_i(\mu_i-\mu_{i+1}).
\end{equation}
\begin{Rem}
The sum of the first~$k$ non-negative Lyapunov exponents is always greater or equal than the speed of any subbundle of the H-N filtration. More precisely, let~$f:X\rightarrow C$ be a semi-stable fibration with associated Lyapunov exponents~$L_f=\sum_{i=1}^g\lambda_i$. Then for any~$1\leq k\leq g$ we have
\begin{equation}
\label{polygon}
\frac{\deg\mathcal{E}_i}{2g_C-2+s}\leq\sum\limits_{i=1}^k\lambda_i,
\end{equation}
by~\cite{lowboundslyapunov}. Of course~\eqref{polygon} is an equality for~$k=g$.
\end{Rem}
The crucial part of Xiao's proof of the slope inequality is now the definition of a 'degree' corresponding to the moving part of~$\mathcal{E}_i$ for~$1\leq i\leq n$. For this reason we let~$\mathcal{L}$ be a sufficiently ample line bundle on~$C$ such that~$\mathcal{E}_i\otimes\mathcal{L}$ is generated by its global sections. We denote by~$\Lambda(\mathcal{E}_i)\subseteq|\omega_f\otimes f^{*}\mathcal{L}|$ the linear subsystem corresponding to sections in~$H^0(C,\mathcal{E}_i\otimes\mathcal{L})$. Now~$Z(\mathcal{E}_i)$ is defined as the fixed part of~$\Lambda(\mathcal{E}_i)$, while~$M(\mathcal{E}_i)=\omega_f-Z(\mathcal{E}_i)$ is the moving part. Note that these definitions do not depend on the choice of the line bundle~$\mathcal{L}$. Let~$F$ be a general fiber of the fibration~$f:X\rightarrow C$, then we define
$$d_i:=M(\mathcal{E}_i).F$$
and set further~$d_{n+1}:=2g-2$. For example we have~$d_n=2g-2$, whereas the dimension of the restricted linear subsystem~$\Lambda(\mathcal{E}_i)|_F$ is~$r_i-1$ (see~{\cite[page~157]{hartshorne}}). The next lemma is the main tool for working with the slope via the H-N filtration.
\begin{Lemma}[{\cite[Lemma~2]{lowslope}}]
\label{lemmaseqofind}
Let~$f:X\rightarrow C$ be a relatively minimal fibration. Then for any sequence of indices~$1\leq i_1<\dots<i_k\leq n$ we have
$$\omega_f^2\geq\sum\limits_{i=1}^{k}(d_{i_j}+d_{i_{j+1}})(\mu_{i_j}-\mu_{i_{j+1}}),$$
where~$i_{k+1}:=n+1$.
\end{Lemma}
\begin{Ex}
\label{bspseqofind}
By taking the trivial sequence of indices~$1\leq n$, we obtain by applying Lemma~\ref{lemmaseqofind} that
$$\omega_f^2\geq(2g-2)(\mu_1-\mu_n)+(4g-4)(\mu_n)=(2g-2)(\mu_1+\mu_n),$$
which implies for~$g=2$ the slope inequality as in Theorem~\ref{slopeinequality}.
\end{Ex}
Our goal is to prove better lower bounds for the slope of non-hyperelliptic fibrations. For that reason we introduce a generalization of being hyperelliptic. We call a genus~$g$ fibration~$f:X\rightarrow C$ a \emph{double cover of type~$(g,\gamma)$}, if there is a (not necessarily smooth) surface~$Y$ and morphisms~$\theta:X\rightarrow Y$ and~$\varphi:Y\rightarrow C$ such that~$\deg(\theta)=2$, the general fiber of~$\varphi$ is of genus~$\gamma$ and the diagram in Figure~\ref{doublecoveroftype} commutes.
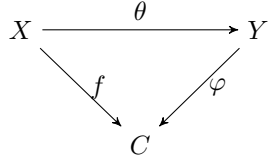
\begin{figure}[!htbp]
\centering
	\begin{tikzpicture}[every node/.style={anchor=center},>=stealth']
	\matrix (m) [matrix of math nodes, row sep=3em, column sep=3em]
	{X &  & Y\\
		 & C &  \\ };
	\draw[->] (m-1-1) -- node[above] {$\theta$} (m-1-3);
	\draw[->] (m-1-1) -- node[right] {$f$} (m-2-2);
	\draw[->] (m-1-3) -- node[right] {$\varphi$} (m-2-2);
	\end{tikzpicture}
\caption{\label{doublecoveroftype}A double cover fibration of type~$(g,\gamma)$, where the general fiber of~$\varphi$ is of genus~$\gamma$.}
\end{figure}
It is clear by this definition that a hyperelliptic fibration is nothing else than a double cover fibration of type~$(g,0)$. The next proposition is  basically the same as~{\cite[Proposition~3.5]{gonality}}. We just adjust a small detail in the proof to improve the lower bound.
\begin{Prop}
\label{propdoublecoveroftype}
Let~$f:X\rightarrow C$ be a relatively minimal fibration of genus $g\geq 3$. If~$f$ is not a double cover fibration or if~$f$ is a double cover fibration of type~$(g,\gamma)$ with~$4\gamma\geq g-1$, then we have
$$\lambda_f\geq\frac{9(g-1)}{2g+1}.$$
\end{Prop}
\begin{proof}
Let~$F$ be a general fiber of~$f$, then we define
$$l:=\min\{i:\Lambda(\mathcal{E}_i)|_F \text{ defines a birational map for } F\}$$
and
$$l':=\min\{i:\Lambda(\mathcal{E}_i)|_F \text{ is a finite map of degree two for } F\}.$$
According to the proof of~{\cite[Theorem 3.2]{gonality}} we may assume that such an~$l'$ exists. It is clear that~$1\leq l'<l\leq n$. We further let~$\gamma$ be the geometric genus of the image of~$F$ under~$\Lambda(\mathcal{E}_i)|_F$. Note that this~$\gamma$ is precisely the same if~$f$ is a double cover of type~$(g,\gamma)$.  For~$l'\leq i\leq l-1$ we define
$$\theta_i:=\left\{
\begin{array}{lll}
 1, & & \text{ if } i=l'=1 \text{ and } r_1=1;\\ 
 \min\{3r_i-3,2r_1+\gamma-1\}, & & \text{ otherwise. } \\
\end{array}
\right.$$
Now we use~{\cite[(3-22)]{gonality}}, which says in particular that
\begin{align}
\omega_f^2  \geq & \sum\limits_{i=1}^{l'-1}\left(\frac{1}{2}(3r_i-2)+\frac{1}{2}(d_i+d_{i+1})\right)(\mu_i-\mu_{i+1}) \nonumber \\
 & + \sum\limits_{i=l'}^{l-1}\left(\frac{1}{2}(2\theta_i-r_i)+\frac{1}{2}(d_i+d_{i+1})\right)(\mu_i-\mu_{i+1}) \label{nonhypboundeq0}\\
 & + \sum\limits_{i=l}^{n}\left(\frac{1}{2}(5r_i-6)+\frac{1}{2}(d_i+d_{i+1})\right)(\mu_i-\mu_{i+1}). \nonumber
\end{align}
We can now apply for the individual terms the following inequalities as stated in~{\cite[page 13\&14]{gonality}}, namely
\begin{align}
\label{nonhypboundeq1}
\frac{1}{2}(3r_i-2)+\frac{1}{2}(d_i+d_{i+1})\geq\frac{9}{2}r_i-2, & & \text{ for } 1\leq i\leq l'-1, \nonumber\\
\frac{1}{2}(2\theta_i-r_i)+\frac{1}{2}(d_i+d_{i+1})\geq\frac{9}{2}r_i-3, & & \text{ for } l'\leq i\leq l-1,
\end{align}
and we can directly compute
\begin{equation}
\label{nonhypboundeq3}
\frac{1}{2}(5r_n-6)+\frac{1}{2}(d_n+d_{n+1})=\frac{9}{2}g-5.
\end{equation}
As for~$l\leq i\leq n-1$ we know that~$\Lambda(\mathcal{E}_i)|_F$ defines a birational map for~$F$, we can apply Castelnuovo's bound as stated in~{\cite[III.§2.]{geomofalgcurvesvol1}}. Therefore we have
$$d_i\geq\frac{g}{m}+\frac{(m+1)r_i}{2}+\frac{m-1}{2}, \text{ where } m=\left\lfloor\frac{d-1}{r-1}\right\rfloor.$$
Since~$2\leq r_i\leq g-1$ this implies in particular that
\begin{equation}
\label{verbesserung}d_i\geq 2r_i
\end{equation}
By rearranging while using~$r_{i+1}\geq r_i$ we get to
\begin{align}
\frac{1}{2}(5r_i-6)+\frac{1}{2}(d_i+d_{i+1})\geq\frac{9}{2}r_i-3, & & \text{ for } l\leq i\leq n-1. \label{nonhypboundeq4}
\end{align}
Now using that
$$\chi_f=\sum\limits_{i=1}^n r_i(\mu_i-\mu_{i+1})$$
by~\eqref{summechif} and putting the inequalities~\eqref{nonhypboundeq1}, \eqref{nonhypboundeq3} and~\eqref{nonhypboundeq4} together with~\eqref{nonhypboundeq0} we arrive at
$$\omega_f^2\geq\frac{9}{2}-2\mu_1-\mu_l'-2\mu_n\geq\frac{9}{2}-3\mu_1-2\mu_n.$$
The last inequality follows from the H-N filtration since~$\mu_1\geq\mu_l'$. Now combining this with Example~\ref{bspseqofind} brings us to
$$\frac{2g+1}{2g-2}\omega_f^2\geq\frac{9}{2}\chi_f+\mu_n\geq\frac{9}{2}\chi_f,$$
which finishes the proof.
\end{proof}
\begin{Rem}
Lu and Zuo used in~{\cite[Proposition~3.5]{gonality}} the inequality given by $d_i\geq 2r_i-1/2$ instead of~$\eqref{verbesserung}$. This then implies as lower bound for the slope that
$$\lambda_f\geq\frac{18(g-1)}{4g+3}.$$
Note that for~$g=6$ their lower bound is the same as the general slope inequality from Theorem~\ref{slopeinequality}.
\end{Rem}
\begin{Cor}
\label{corlowboundnonhyp}
Let~$f:X\rightarrow C$ be a non-hyperelliptic fibration of genus~$g\geq 3$. Then we have
$$\lambda_f\geq\frac{9(g-1)}{2g+1} \text{ for }3\leq g\leq 12,$$
and
$$\lambda_f\geq 4 \text{ for } g>12.$$
\end{Cor}
\begin{proof}
If~$f$ is not a double cover of type~$(g,\gamma)$ with ~$4\gamma< g-1$, then we are done by Proposition~\ref{propdoublecoveroftype}. So let us assume the contrary. Then by {\cite[Theorem 3.1]{slopedoublecover}} we know that
$$\lambda_f\geq\frac{4(g-1)}{g-\gamma}\geq 4.$$
Since we assume~$f$ to be non-hyperelliptic, we conclude with~$\gamma\geq 1$.
\end{proof}
Note that the bound of Corollary~\ref{corlowboundnonhyp} is weaker than~\eqref{lowboundsnonhyp} but beats the classical slope inequality in Theorem~\ref{slopeinequality} for~$g\geq 6$.
\section{Speed of semi-stable fibrations}
Beside the slope, there is another relative invariants linked to a fibration in this paper that is playing a central role. The, for us, most important ratio is the \emph{speed}
$$ L_f \,=\, \frac{2\deg f_{*}\omega_{f}}{2g_{C}-2+s},$$
where~$f:X\rightarrow C$ is a semi-stable fibration of fiber genus~$g\geq 2$.
We recall from Section~\ref{sec:lyapunov} that~$L_f$ is exactly the sum of non-negative Lyapunov exponents, which was the reason for naming it speed.
As described in the introduction is the speed bounded from above by the strict Arakelov inequality~$L_f< g$. We will give a proof of this inequality in Theorem~\ref{arakelovineq}.
\begin{Rem}
\label{speednoinvariant}
For a semi-stable fibration~$f:X\rightarrow C$ let us consider the embedded curve~$\phi_C:C\hookrightarrow\M_g$. We can then calculate its speed by
$$L_f=2\frac{\deg\phi_C^*\lambda_g}{2g_C-2+s},$$
where~$\lambda_g$ is the Hodge class of~$\M_g$. Here the number s above is not an intersection number, but is the set-theoretic intersection of~$C$ with the boundary~$\M_g\setminus\mathcal{M}_g$. For this reason the speed is not an invariant for classes of curves in the moduli space of stable curves, in contrast to the slope. Indeed, consider for example two plane quartics~$F$ and~$G$. The pencil~$aF+bG$ for~$[a:b]\in\P^1$ gives (after blowing up at the base points) a curve~$C$ in~$\M_3$ with intersection~$3$ and~$27$ with~$\lambda_g$ and~$\delta_0$ respectively (see~{\cite[page~170]{moduli}}). If~$F$ and~$G$ are general then we have that~$s=27$, which is the degree of the discriminant locus. Actually, as speed we therefore calculate
$$L_f=\frac{2\cdot 3}{2\cdot 0-2+27}=\frac{6}{25}.$$
If we choose~$F$ general and~$G$ to have more than one irreducible node then the resulting curve will have the same curve class, hence the same intersection number with the boundary divisor~$\delta_0$, but in this case we have~$s<27$ and therefore a different speed.
\end{Rem}

\subsection{Types of nodes}
\label{singfibers}
Let~$f:X\rightarrow C$ be a semi-stable fibration. A singular point of~$F$ is called a \emph{separating node}, if the partial normalization at this point consists of two components. We further denote with~$\mathcal{Y}_{nc}\rightarrow\ \Delta_{nc}$ all singular fibers that do contain a non-separating node and define~$s_{nc}:=|\Delta_{nc}|$. We will call such singular fibers of~\emph{non-compact type}. Vice versa we define~$\mathcal{Y}_{ct}\rightarrow\ \Delta_{ct}$ to be all singular fibers of \emph{compact type}, i.e.\ fibers that contain no nodes but separating ones and denote~$s_{ct}:=|\Delta_{ct}|$. Let us denote with~$\delta_0(F)$ the number of all non-separating nodes in~$F$. We say that a separating node is \emph{of type~$i$}, if the partial normalization at this point consists of two components of arithmetic genus~$i$ and~$g-i$ (see Figure~\ref{nodetype}) and denote with~$\delta_i(F)$ the number of all separating nodes of type~$i$ in~$F$. We additionally set~$\delta(F):=\sum \delta_i(F)$ to be the number of all nodes in the singular fiber~$F$.

\begin{figure}[!htbp]
\centering
\begin{tikzpicture}[>=stealth,scale=0.8] 

\begin{scope}
\coordinate (P1) at (0,0);
\coordinate (P2) at (2,1.5);
\coordinate (P3) at (5,1.5);
\coordinate (P4) at (7,0);
\coordinate (P5) at (5,-1.5);
\coordinate (P6) at (2,-1.5);
\coordinate (P7) at (9,1.5);
\coordinate (P8) at (10,1.5);
\coordinate (P9) at (12,0);
\coordinate (P10) at (10,-1.5);
\coordinate (P11) at (9,-1.5);

\path[draw,name path=P1--P1] (P1) to[out=90,in=180] (P2) 
to[out=0,in=180] (P3)						      
to[out=0,in=90] (P4)
to[out=270,in=0] (P5)
to[out=180,in=0] (P6)
to[out=180,in=270] (P1);
\path[draw,name path=P4--P4] (P4) to[out=90,in=180] (P7) 
to[out=0,in=180] (P8)						      
to[out=0,in=90] (P9)
to[out=270,in=0] (P10)
to[out=180,in=0] (P11)
to[out=180,in=270] (P4);

\fill[black] (P4) circle (2pt);

\draw[] (2,-.05) arc(0:180:.65cm and .30cm) (2.1,.15) arc(0:-180:.75cm and .4cm);
\begin{scope}[xshift=4cm,yshift=0cm]
\draw[] (2.4,-.05) arc(0:180:.65cm and .30cm) (2.5,.15) arc(0:-180:.75cm and .4cm);
\end{scope}
\begin{scope}[xshift=1.5cm,yshift=0cm]
\draw[] (2.4,-.05) arc(0:180:.65cm and .30cm) (2.5,.15) arc(0:-180:.75cm and .4cm);
\end{scope}
\begin{scope}[xshift=6.5cm,yshift=0cm]
\draw[] (2.4,-.05) arc(0:180:.65cm and .30cm) (2.5,.15) arc(0:-180:.75cm and .4cm);
\end{scope}
\begin{scope}[xshift=9cm,yshift=0cm]
\draw[] (2.4,-.05) arc(0:180:.65cm and .30cm) (2.5,.15) arc(0:-180:.75cm and .4cm);
\end{scope}
\end{scope}
\node[draw=white] at (4.5,0) {$\dots$};
\node[draw=white] at (9.5,0) {$\dots$};
\node[draw=white] at (4,-1.8) {genus $g-i$};
\node[draw=white] at (9.5,-1.8) {genus $i$};
\node[draw=white] at (7,0.8) {$p$};

\end{tikzpicture}
\caption{\label{nodetype}A separating node~$p$ of type~$i\geq 1$.}
\end{figure}
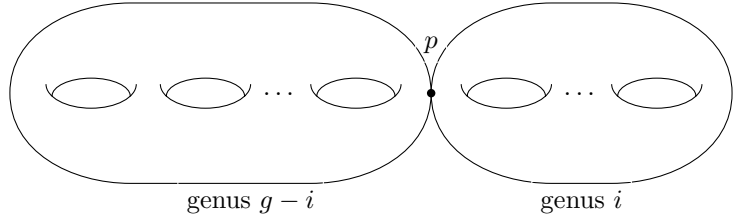

Assume that~$F$ consists of~$l(F)$ irreducible components such that $F=\sum_{i=1}^{l(F)}F_i$. Then we recall that the \emph{geometric genus} of~$F$ is defined as the sum of the geometric genera of the components, i.e.\
$$g_{geo}(F):=\sum_{i=1}^{l(F)}g(F_i'),$$
where~$F_i'$ is the normalization of~$F_i$.
\begin{Lemma}[{\cite[Lemma 2.2]{minnumber}}]
	\label{Lemmasingfibers}
	Let~$F$ be a semi-stable fiber of a relatively minimal fibration~$f:X\to C$ of genus~$g$. Then we have
	$$\delta(F)=g-g_{geo}(F)+l(F)-1.$$
\end{Lemma}
Using this lemma we easily see the equivalences
\begin{equation*}
\begin{aligned}
 g_{geo}(F)=g & \Longleftrightarrow & \delta(F)=l(F)-1 & \Longleftrightarrow & \delta_0(F)=0 & \Longleftrightarrow & F\in\mathcal{Y}_{ct}.
\end{aligned}
\end{equation*}
For the fibration~$f:X\rightarrow C$ we further define
$$\delta_0(f):=\sum_{t\in C}\delta_0(F_t) \text{ and } \delta_i(f):=\sum_{t\in C}\delta_i(F_t).$$
Note that
$$\delta_f=\delta_0(f)+\delta_1(f)+\sum_{i=1}^{\lfloor g/2\rfloor}\delta_i(f).$$
In terms of the moduli space of stable curves where~$\phi_C:C\hookrightarrow\M_g$ is the embedding induced by~$f:X\rightarrow C$ we get
$$\delta_0(f)=\deg(\phi_C^*(\delta_0)) \text{ and } \delta_i(f)=\deg(\phi_C^*(\delta_i)).$$

\subsection{The strict Arakelov inequality}
\label{sectionarakelov}
The Arakelov inequality was first proven in~\cite{arakelov} in terms of heights of sections. We consider the \emph{strict Arakelov inequality} for semi-stable fibrations as established by Sheng-Li Tan.
\begin{Thm}[{\cite[Theorem 2]{tanminnumber}}]
	\label{arakelovineq}
	Let~$f:X\to C$ be a semi-stable fibration of genus~$g\geq 2$. Then we have
	\begin{equation*}
	L_f=2\frac{\deg f_*\omega_f}{2g_C-2+s}<g.
	\end{equation*}
\end{Thm}
Let us start with a semi-stable fibration~$f:X\rightarrow C$. It is a result of Arake\-lov~\cite{arakelov} and Beauville~\cite{beauville} that~$\omega_{X/C}$ is a nef divisor and all curves~$E$ in~$X$ such that~$\omega_{X/C}.E=0$ are the~$(-2)$-curves in the fibers. Contracting all~$(-2)$ curves defines a not necessary semi-stable, but rather a stable fibration~$f^{\#}:X^{\#}\rightarrow C$, since every fiber is stable. We also call~$f^{\#}$ the \emph{stable model of~$f$}. The surface~$X^{\#}$ might now be singular, i.e.\ contain singular points~$q$, where~$q$ is a rational double point of type~$A_{m_q}$ (see Proposition~\ref{propdoublepoints}). We know that~$m_q$ is exactly the number of $(-2)$-curves in~$X$ over~$q$. If~$q$ is a singular point on the smooth part of~$X^{\#}$, then we set~$m_q=0$.
\par
An, for us, important invariant is the rational number~$r_f$, which we define in terms of the stable model~$X^{\#}$ as
$$r_f:=\sum\limits_{q\in X^{\#}}\frac{1}{m_q+1}.$$
It is immediately obvious that~$r_f\leq\delta_f$.
\begin{Rem}
\label{boundrf}
Let~$f:X\rightarrow C$ be a semi-stable fibration of genus~$g$.
We show that~$r_f$ is bounded by~$r_f\leq (3g-3)s$, where~$s$ is the number of singular fibers. In fact, we have for each singular fiber~$F$ that
\begin{equation}
\label{nodesgenustwo}\delta(F)=g-g_{geo}(F)+l(F)-1
\end{equation}
by Lemma~\ref{Lemmasingfibers}. Additionally, the lemma states that the fiber has as geometric genus~$g_{geo}(F)=g$ if and only if~$F$ is of compact type. Note that equation~$\eqref{nodesgenustwo}$ still holds for the stable fiber~$F^{\#}$ after contracting all $(-2)$-curves as we subtract on both sides the same amount in this procedure, hence we have
$$\delta(F^{\#})=g-g_{geo}(F)+l(F^{\#})-1.$$
This implies that the number~$\delta(F^{\#})$ of singular points of the stable model is maximal for a geometric genus of~$g_{geo}(F)=0$. More concretely, as any rational component of~$F^{\#}$ must intersect the other components in at least three points, we have~$\delta(F^{\#})\leq 3g-3$ (see also~{\cite[page~51]{moduli}}). Let~$r(F)$ be the contribution of each singular fiber to~$r_f$, i.e.\ 
$$r_f=\sum\limits_{i=1}^{s}r(F)$$
By our definition~$r(F)\leq\delta(F^{\#})$. We therefore conclude trivially~$r_f\leq (3g-3)s$.
\end{Rem}
A rational double point~$q$ on~$X^{\#}$ is isomorphic to~$\CC^2/G_q$, where~$G_q$ is a finite subgroup of~$GL(2,\CC)$, with~$0$ as only fixed point. We call a singularity given by~$\CC^2/G_q$ a \emph{quotient singularity}, i.e.\ every rational double point is a quotient singularity. Let~$\widetilde{X}_q$ be the minimal resolution of this quotient singularity with the exceptional divisor~$E_q$. We put
$$\nu(q):=e(E_q)-\frac{1}{|G_q|},$$
where~$e(E_q)$ is the topological Euler characteristic. Miyaoka's inequality now states the following.

\begin{Thm}[{\cite[Corollary 1.3]{miyaoka}}]
\label{miyaoka}
	Let~$X^{\#}$ be a complex surface such that~$K_{X^{\#}}$ is nef and that~$X^{\#}$ has only rational double points as singularities. Let~$\sigma:X\rightarrow X^{\#}$ be the minimal resolution. Then
	$$3\sum_{q\in X^{\#}} \nu(q)\leq 3c_2(X)-c_1^2(X).$$
\end{Thm}
Assume that on our surface~$X^{\#}$ a quotient singularity~$q$ is isomorphic to a rational double point of type~$A_{m_q}$, for some $m_q\geq 2$. Then we immediately get
$$\nu(q)=(m_q+1)-\frac{1}{m_q+1}.$$
We note that if we consider additionally the nodes of the smooth part in~$X^{\#}$, the inequality in Theorem~\ref{miyaoka} does not change. Therefore we can calculate the number of singular points in the singular fibers via
$$
	\delta_f=\sum_{q\in X^{\#}}(m_q+1).
$$
and we recall that
$$r_f=\sum_{q\in X^{\#}}\frac{1}{m_q+1},$$
which implies the following corollary.
\begin{Cor}
	\label{Mineq}
Let~$f:X\rightarrow C$ be a semi-stable fibration of genus~$g$, where the canonical divisor~$K_X$ is nef. Then we have
	$$3(\delta_f-r_f)\leq 3c_2(X)-c_1^2(X).$$
\end{Cor}
Now we can prove an upper bound for the self-intersection~$\omega_f^2$.
\begin{Lemma}
	\label{ineq} Let~$f:X\to C$ be a semi-stable fibration of genus~$g$. Then for any~$n\in\NN$ satisfying $2g_C-2+\frac{n-1}{n}s\geq 0$, we have
\begin{equation*}
	\omega_{f}^2\leq(2g-2)(2g_C-2+s)+\frac{3 r_f}{n^2}-\frac{(2g-2)s}{n}.
\end{equation*}
\end{Lemma}
\begin{proof}
	Let~$\pi:\widetilde{C}\rightarrow C$ a base change of degree~$d\cdot n$ totally ramified over~$\Delta$ with ramification index~$n$. Let~$\widetilde{f}:\widetilde{X}\rightarrow\widetilde{C}$ be the pullback fibration of~$f$ with respect to~$\pi$, i.e.\ the smooth minimal model of the fiber product~$X\times_C \widetilde{C}\rightarrow \widetilde{C}$. By Riemann-Hurwitz, we know that
	$2g_{\widetilde{C}}-2=dn\cdot(2g_C-2+\frac{n-1}{n}s)$ and therefore by our assumption~$g_{\widetilde{C}}>0$. Hence~$\widetilde{X}$ is of general type and in particular the canonical divisor
	$$K_{\widetilde{X}}\sim K_{\widetilde{X}/\widetilde{C}}+K_{\widetilde{X}}|_{\widetilde{F}}$$
	is nef. Therefore using~Corollary~\ref{Mineq} we get
	$$\omega_{\widetilde{f}}\leq3c_2(\widetilde{X})-3\delta_{\widetilde{f}}+3r_{\widetilde{f}}-2(2g_{\widetilde{C}}-2)(2g-2),$$
	or equivalently
	$$\omega_{\widetilde{f}}\leq 3r_{\widetilde{f}}+(2g_{\widetilde{C}}-2)(2g-2).$$
	Note that for any node~$q$ in a singular fiber, the preimage~$\Pi^{-1}(q)$ consists of~$d$ disjoint curves of type~$A_{n-1}$, hence~$m_{\widetilde{q}}+1=n(m_q+1)$ and~$r_{\widetilde{f}}=\frac{d}{n}r_f$.
Since all fibrations are semi-stable, the base change of degree~$dn$ implies that~$\omega_{\widetilde{f}}^2=dn\cdot\omega_f^2$. So in total we get
$$\omega_f^2\leq(2g-2)(2g_C-2+\frac{n-1}{n}s)+\frac{3 r_f}{n^2}.$$
and therefore Lemma~\ref{ineq}.
\end{proof}
\begin{Rem}
\label{bemn>1}
The condition $2g_C-2+\frac{n-1}{n}s_{nc}\geq 0$ is in fact always satisfied for~$n>1$, if~$s_{nc}>0$. In the case~$g_C=0$ we know by~\cite{minnumber} that $s_{nc}\geq4$.
\end{Rem}

\begin{Rem}
	\label{canonicalclassineq}
	From Lemma~\ref{ineq} we can directly deduce the strict \emph{canonical class inequality}
	$$\omega_f^2 < (2g-2)(2g_C-2+s).$$
\end{Rem}
 We therefore have found a proof of the strict Arakelov inequality.
\begin{proof}[Proof of Theorem~\ref{arakelovineq}]
	We use Lemma~\ref{ineq}, which we combine additionally with the slope inequality of Theorem~\ref{slopeinequality}. Then we have
	$$\frac{4(g-1)}{g}\chi_f\leq(2g-2)(2g_C-2+s)+\frac{3 r_f}{n^2}-\frac{(2g-2)s}{n}$$
	or rearranged
	$$2\chi_f\leq g(2g_C-2+s)+\frac{g}{2g-2}\left(\frac{3 r_f}{n^2}-\frac{(2g-2)s}{n}\right).$$
Since
$$\lim\limits_{n\rightarrow\infty}\left(\frac{3 r_f}{n^2}-\frac{(2g-2)s}{n}\right)<0,$$
we have proven the strict inequality for~$s>0$. If~$s=0$ than we will see in Proposition~\ref{surfacebundle} that a much stronger inequality holds.
\end{proof}
\begin{Rem}
\label{minnummbersingfib}
For a semi-stable fibration~$f:X\rightarrow\P^1$ over the projective line, we conclude directly that there are always at least five singular fibers which was proven first in~\cite{tanminnumber}. Since~$g_{\P^1}=0$, we have
$$\chi_f=\chi(\mathcal{O}_X)+(g-1)\geq g.$$
But by the strict Arakelov inequality we know that~$2\chi_f< g(-2+s)$ and therefore~$s\geq 5$. This bound is in fact sharp, in Example~\ref{bspbeauville} we will get in touch with a~$g=3$ fibration with five singular fibers.
\end{Rem}

\subsection{Bounds of the speed}
\label{improvements}
After having introduced the strict Arakelov inequality~$L_f<g$ for semi-stable fibrations in Section~\ref{sectionarakelov}, which yields an upper bound for the speed, we ask the question if this bound is the best possible.
\subsubsection{Kodaira fibrations}
Let us consider a smooth fibration~$f:X\to C$, i.e.\ a fibration with no singular fibers. They were originally constructed by Kodaira (see~\cite{kodairaelliptic1} and~\cite{kodaira}) as examples to show that the signature of the surface is not multiplicative in the fiber, a result that was somewhat surprising, since the multiplicity holds for the topological Euler characteristic. In his honor they are nowadays called \emph{Kodaira fibration}. In the literature one can also find the name \emph{surface bundle}. We know that a Kodaira fibration is iso-trivial if the fiber genus is~$g\leq 2$. Indeed, since~$\mathcal{M}_2$ is affine (see~\cite{igusa}), any embedding of a complete curve in~$\mathcal{M}_2$ must be constant. Additionally the base genus has to be~$g_C\geq 2$ since we have by~{\cite[Proposition III.11.4]{compactcomplex}} that
$$c_2(X)=4(g-1)(g_C-1)>0.$$
For Kodaira fibrations we can show a stronger upper bound than the Arakelov inequality. This result should be known to experts, but since we are not aware of any reference we prove it here.

\propsurfacebundle*
\begin{proof}
Using the multiplicity of the Euler characteristic we get
\begin{equation*}
\chi_f=\chi(\mathcal{O}_X)-\frac{1}{4}e(F)e(C)=\chi(\mathcal{O}_X)-\frac{1}{4}c_2(X)=\frac{1}{12}(c_1^2(X)-2c_2(X)).
\end{equation*}
As~$X$ is of general type the Bogomolov-Miyaoka-Yau inequality~$c_1^2(X)\leq 3c_2(X)$ therefore defines an upper bound
$$\chi_f\leq\frac{1}{12}c_2(X)=\frac{1}{12}(2g-2)(2g_C-2)=\frac{g-1}{6}(2g_C-2)$$
and the claim follows.
\end{proof}

\subsubsection{Fibrations of low base genus}
The strictness of the Arakelov inequality or more precisely the strictness of the canonical class inequality yields an upper bound for the speed of a semi-stable fibration, if the quotient
$$\frac{s}{2g_C-2+s}$$
is bounded from below. Of course this is the case if the genus of the base curve~$C$ is low.

\lowbasegen*
\begin{proof}
	By Remark~\ref{boundrf} we get that
	$$r_f=\sum_{q\in X^{\#}}\frac{1}{n_q+1}\leq (3g-3)s.$$
	Now we will use Lemma~\ref{ineq}, which tells us that
	$$\omega_f^2\leq 2(2g_C-2+s)-\frac{2n-9}{n^2}\cdot s(g-1).$$
	Since~$4(g-1)\chi_f\leq g\cdot\omega_f^2$ by the slope inequality of Theorem~\ref{slopeinequality}, we get in terms of the speed
	$$L_f\leq g-\frac{2n-9}{n^2}\cdot\frac{g}{2}\cdot\frac{s}{2g_C-2+s}\leq g-\frac{2n-9}{n^2}\cdot\frac{g}{2}\cdot\frac{1}{m},$$
	for every~$n>1$ (see also Remark~\ref{bemn>1}). The minimum is attained for~$n=9$ and hence the claim follows.
\end{proof}
For fibrations over the projective line we can easily get an upper bound for the speed that depends on the number of singular fibers.
\begin{Lemma}
\label{lemmasklein}
Let~$f:X\rightarrow \P^1$ be a semi-stable genus~$g$ fibration over the projective line with~$s$ singular fibers. Let~$g$ or~$s$ be even. Then the speed~$L_f$ is bounded by
$$L_f\leq g-\frac{2}{s-2}.$$
\end{Lemma}
\begin{proof}
Using the strict Arakelov inequality (Theorem~\ref{arakelovineq}) we can immediately deduce that~$\chi_f<g(s-2)/2=(gs)/2-1$. But since~$\chi_f$ is by definition the degree of a vector bundle on~$\P^1$ and therefore a natural number and since~$gs$ is even, we have in fact~$\chi_f\leq gs/2-2$, which computes the upper bound of the speed.
\end{proof}
\begin{Rem}
In particular, Lemma~\ref{lemmasklein} states that
$$L_f\leq \frac{4}{3} \text{ for~$s=5$ and } L_f\leq \frac{3}{2} \text{ for } s=6.$$
We recall that~$s\leq 4$ is not possible by Remark~\ref{minnummbersingfib}.
\end{Rem}

\subsubsection{Non-hyperelliptic fibrations}
Let~$f:X\to C$ be a relatively minimal fibration such that the general fiber~$F$ is non-hyperelliptic. We already discussed that the slope of such a fibration does not reach the lower slope bound and hence, not surprisingly, we can find an upper bound for the speed for small fiber genus. Moreover low slope and high speed are in fact related.

\speednonhyp*
\begin{proof}
The claim follows from the lower slope bounds for non-hyperelliptic fibrations in~\eqref{lowboundsnonhyp} and Corollary~\ref{corlowboundnonhyp} as well as from the strict canonical class inequality (see Remark~\ref{canonicalclassineq}).
\end{proof}

\section{Hyperelliptic fibrations}
\label{sechypfib}
A \emph{hyperelliptic fibration} is a fibration, whose general fiber is a hyperelliptic curve of genus~$g$, i.e.\ there is a hyperelliptic involution~$\sigma$ acting on each fiber. In this section we introduce some techniques to construct examples of hyperelliptic fibrations.
\par
An involution with no isolated fixed points on a surface can equivalently be viewed as a branched double cover on a ruled surface. Following this approach, we will construct a genus~$g$ datum for a hyperelliptic fibration. Working with this detour over ruled surfaces is going back to the fundamental work of Horikawa (see~\cite{horikawa} and also~\cite{persson}). Xiao~\cite{surffibr} developed his techniques much further. For this approach we will introduce an even divisor on a ruled surface, which defines a double cover that determines the hyperelliptic fibration. Such a divisor can contain singularities that are resolved by an even blow-up, which ensures that the strict transform is again an even divisor and defines therefore a double cover. Knowing the singularities and the bidegree of the divisor is enough information to calculate the slope of the fibration.

\subsection{Ruled surfaces}
\label{chapterruledsurf}
Let us consider a smooth ruled surface~$P$ over a smooth curve~$C$, which means that we have a fibration~$\varphi:P\rightarrow C$ where the fibers are just~$\P^1$'s. We will follow here the notation of~\cite{hartshorne} and always remember that by mentioning~$P$, the morphism~$\varphi$ and the base~$C$ are part of our datum. Any ruled surface is given by a rank~$2$ vector bundle~$\mathcal{E}$ on~$C$, such that~$P=\P(\mathcal{E})$. But as the vector bundle~$\mathcal{E}$ is only unique up to tensoring with a line bundle, we want our vector bundle to be of a 'minimal' degree. That is we call a vector bundle \emph{normalized}, if it has a nonzero section, but~$\mathcal{E}\otimes\mathcal{L}$ has none for every line bundle~$\mathcal{L}$ of negative degree. So let~$\mathcal{E}$ always be a normalized vector bundle with~$P=\P(\mathcal{E})$. The integer~$e:=-\deg\mathcal{E}$ is then an invariant of the ruled surface~$P$ (see {\cite[Proposition V.2.8]{hartshorne}}). The nonzero section of~$\mathcal{E}$ gives rise to a horizontal curve~$C_0$ on~$P$ with~$C_0^2=-e$. Since~$\mathcal{E}$ is normalized, the number~$-e$ is the smallest self-intersection number of any horizontal curve. Now~$C_0$ and~$\Gamma$, where~$\Gamma$ is a general fiber of~$\varphi$, generate the N{\'e}ron-Severi group of the ruled surface~$P$ (see {\cite[Proposition V.2.3]{hartshorne}}). So as any divisor~$D$ on~$P$ is linear equivalent to~$aC_0+b\Gamma$, we will also say that~$D$ has \emph{bidegree}~$(a,b)$. For instance, the canonical divisor of~$P$ has bidegree~$(-2,2g_C-2-e)$ {\cite[Corollary V.2.11]{hartshorne}} and therefore we can compute~$K_P^2=8(1-g_C)$.
\par
Let~$D$ be a divisor on the ruled surface~$P$ given locally at a point~$p=(y_0,z_0)\in D$ by a function~$f(y_0,z_0)=0$, where
$$f(y,z)=\sum_{j,k\geq 0}a_{jk}y^j z^k.$$
Then we define the \emph{multiplicity of~$p$ in~$D$} as
$$\rm{mult}_p(D):=m_p=\inf_{j,k}\{j+k|a_{jk}\neq 0\}.$$
It is clear from this definition that~$m_p\geq 1$. In fact, if~$m_p=1$, then the point~$p$ is smooth on~$D$.

\subsection{Genus g datum}

Our interest in ruled surfaces arises from the hyperelliptic involution of a hyperelliptic genus~$g$ fibration~$f:X\rightarrow C$. In fact, after blowing up a finite number of points, every hyperelliptic fibration factors through a ruled surface~$P$ by a double cover. The aim of this section is to define a datum on a ruled surfaces which determines a hyperelliptic fibration. Then the slope of this fibration can be calculated from information arising from the datum.
\par
Let~$\mathcal{L}$ be an invertible sheaf and~$R$ an effective, reduced divisor on a minimal, ruled surface~$\varphi:P\rightarrow C$ such that~$\mathcal{O}_P(R)\cong\mathcal{L}^{\otimes 2}$. We call the tuple~$(R,\mathcal{L})$ a \emph{double cover datum} on~$P$. Let~$s\in H^0(P,\mathcal{L}^{\otimes 2})$ then the dual~$s^{\vee}$ defines an $\mathcal{O}_P$-algebra structure on
$$\mathcal{A}:=\mathcal{O}_P\oplus\mathcal{L}^{-1}.$$
Setting~$X={\rm Spec}(\mathcal{A})$ induces now a double cover~$\theta:X\rightarrow P$. The surface~$X$ might have singular points arising over singular points of~$R$. In fact we have~$X$ smooth if and only if~$R$ is smooth. For more information we refer to~\cite{prootcover}. Note that the reduced divisor~$R$, called \emph{branch divisor},  must be even, where we recall that a divisor on~$P$ is called \emph{even}, if it is divisible by two in the N{\'e}ron-Severi group of~$P$.
\par
We consider a double cover datum~$(R,\mathcal{L})$ on~$P$, where the branch divisor~$R$ intersects a general fiber in~$2g+2$ points. If~$R$ is smooth, then we can directly construct the double cover~$\theta:X\rightarrow P$ branched over~$R$. As~$P$ is minimal, there are no curves whose preimage is a~$(-1)$-curve in~$X$. Hence the composition of~$\theta$ with the ruling of~$P$, i.e.~$\varphi\circ\theta:X\rightarrow P$ is then a relatively minimal hyperelliptic fibration of genus~$g$.
\par
So let us in contrary assume that~$R$ is not smooth. Then there is a unique resolution $\widetilde{\psi}:\widetilde{P}\rightarrow P$ of the singularities, which we construct in a 'even' way. More concretely we want to produce a double cover datum~$(\widetilde{R},\mathcal{L})$ on~$\widetilde{P}$.
\begin{Defi}
\label{minresolution}
We call a sequence of blow-ups
\begin{equation*}
\widetilde{\psi}=\psi_1\circ...\circ\psi_r: (\widetilde{P},\widetilde{R})=(P_r,R_r)\geht^{\psi_r}...\geht^{\psi_1}(P_0,R_0)=(P,R) 
\end{equation*}
an \emph{even resolution of~$R$}, if~$\widetilde{R}$ is a smooth reduced even divisor and~$R_i$ is the reduced even inverse image of~$R_{i-1}$ under~$\psi_i$, i.e.\
$$R_i=\psi_i^*R_{i-1}-2\left\lfloor\dfrac{m_{p_i}}{2}\right\rfloor E_i.$$
We will call~$\widetilde{\psi}$ a \emph{minimal even resolution of~$R$}, if for any even resolution~$\widetilde{\psi}':\widetilde{P}'\rightarrow P$ of~$R$ there exists an morphism~$\widetilde{\Psi}:\widetilde{P}'\rightarrow\widetilde{P}$ such that~$\widetilde{\psi}'=\widetilde{\psi}\circ\widetilde{\Psi}$ and~$\widetilde{\Psi}(\widetilde{R}')=\widetilde{R}$.
\end{Defi}
Note that a minimal even resolution of~$R$ is by our definition unique up to isomorphisms of~$\widetilde{P}$. Indeed, if there is another minimal even resolution~$\widetilde{P}'\to P$ there exists~$\widetilde{\Psi}:\widetilde{P}'\rightarrow\widetilde{P}$ and~~$\widetilde{\Psi}':\widetilde{P}\rightarrow\widetilde{P}'$ such that~$\widetilde{\Psi}\circ\widetilde{\Psi}'=\rm{id}_P$ and~$\widetilde{\Psi}'\circ\widetilde{\Psi}=\rm{id}_{P'}$, hence~$P\cong P'$.
\begin{Prop}[\cite{Enokizono}]
Let~$(R,\mathcal{L})$ be a double cover datum on the minimal surface~$P$. Then there exists a minimal even resolution of~$R$, denoted by~$\widetilde{\psi}:\widetilde{P}\rightarrow P$. Moreover there exist~$\widetilde{\mathcal{L}}\in{\rm Pic}(\widetilde{P})$ such that~$\widetilde{\mathcal{L}}^{\otimes 2}\cong\mathcal{O}_{\widetilde{P}}(\widetilde{R})$.
\end{Prop}
\begin{proof}
We are going to construct such a minimal even resolution directly. For this purpose we let~$\psi_1:P_1\to P$ be a blow up of a singular point~$p_1\in R_0:=R$ and we set~$\mathcal{L}_0:=\mathcal{L}$. We define the divisor~$R_1$ on~$P_1$ to be
$$R_1:=\psi_i^*R_0-2\left\lfloor\dfrac{m_{p_1}}{2}\right\rfloor E_1.$$
Here~$E_1$ is the exceptional divisor of the blow-up of~$p_1$. Note that~$R_1$ is the reduced even inverse image of~$R$. Now, on~$P_1$, if there is a singular point~$p_2\in R_1$ we blow up~$p_2$ again by~$\psi_2:P_2\rightarrow P_1$ and continue this procedure inductively until the divisor~$\widetilde{R}:=R_r$ on~$\widetilde{P}:=P_r$ is smooth.
In each step we blow up the singular point~$p_i$ on~$R_{i-1}\subseteq P_{i-1}$ and set
$$\widetilde{\psi}:=\psi_1\circ\dots\circ\psi_r.$$
We define again as reduced even inverse image of~$R_{i-1}$ the divisor~$R_i$ on~$P_i$ where
$$
R_i:=\psi_i^*R_{i-1}-2\left\lfloor\dfrac{m_{p_i}}{2}\right\rfloor E_i
$$
for~$i=1,\dots,r$. Here, consistently, we denote by~$E_i$ the exceptional divisor of the blow-up of~$p_i$. This even resolution is minimal, since each~$\psi_i$ is a blow up centered at a singular point of~$R_{i-1}$. It is easy to see that for every~$i$ there exists a linebundle~$\mathcal{L}_i\in{\rm Pic}(P_i)$ such that~$\mathcal{L}_i^{\otimes 2}\cong\mathcal{O}_{P_i}(R_i)$. In conclusion we finish the proof by setting~$\widetilde{\mathcal{L}}:=\mathcal{L}_r$.
\end{proof}

We finally have a unique double cover datum given by~$(\widetilde{R},\widetilde{\mathcal{L}})$ on the ruled surface $\widetilde{\varphi}:\widetilde{P}\rightarrow P$ such that~$(\widetilde{\psi})_*(\widetilde{R})=R$. The datum defines a double cover whose minimal model is a hyperelliptic genus~$g$ fibration. This allows us to define a triple~$(P,R,\mathcal{L})$ which yields a hyperelliptic fibration~$f:X\rightarrow C$. We say that a singularity on~$R$ is a \emph{simple singularity}, if it is either a double point or a triple point that becomes at most a double point after evenly blowing it up.
\begin{Defi}[\cite{piell}]
\label{defgenusgdatum}
Let~$\varphi:P\rightarrow C$ be a ruled surface over a smooth curve~$C$ and let~$(R,\mathcal{L})$ be a double cover datum on~$P$. Let~$e$ be the minimal self-intersection of a horizontal curve on~$P$ and~$n\in\Z$. Then we call~$(P,R,\mathcal{L})$ a \emph{genus~$g$ datum}, if
\begin{itemize}
\item[i)] we have~$R.\Gamma=2g+2$ for~$\Gamma$ a general fiber of~$\varphi$,
\item[ii)]  there exists an even blow-up~$\hat{P}\rightarrow P$ such that the reduced even inverse image~$\hat{R}$ of~$R$ has only simple singularities.
\end{itemize}
Furthermore we take~$n\in\Z$ to be the number such that the branch divisor~$R$ has as bidegree~$(2g+2,(g+1)e+n)$ in~$P$.
\end{Defi}
The divisor~$R$ is a branch divisor, hence must be even or in other words divisible by two in the N{\'e}ron-Severi group of~$P$. This implies that~$(g+1)e+n$ must be even.
\begin{Rem}
\label{boundsofe}
The number~$e$ of the surface~$P$ is bounded by~$e\leq n/(g+1)$.
Let us consider a genus~$g$ datum~$(P,R,\mathcal{L})$. The ruled surface comes with the base genus~$g_C$ and the invariant~$e$, which is the negative self-intersection number of~$C_0$. Now~$n\in\Z$ is the number such that~$R$ has bidegree~$(2g+2,(g+1)e+n)$, hence~$(g-1)e\equiv n \mod 2$. If the section~$C_0$ of the ruled surface does not belong to~$R$ then it is intersecting the branch divisor properly, therefore~$R.C_0=n-(g+1)e\geq 0$ and hence~$e\leq n/(g+1)$. Otherwise, since~$R$ is a branch divisor which has no multiple components, the curve~$C_0$ intersects~$R\smallsetminus C_0$ properly, i.e.\ $(R\smallsetminus C_0).C_0=n-ge\geq 0$, hence~$e\leq n/g$. Moreover, if~$e>n/(g+1)$, then this forces~$C_0$ to be a component of~$R$.
\end{Rem}

\subsection{Invariants}

So far we have seen that we only need to consider the blow-up of all non-negligible points to calculate the slope. Certain of these singularities correspond exactly to~$(-1)$-curves in the surface~$\widetilde{X}$ or in other words to isolated fixed points of the hyperelliptic involution of~$X$. Given a genus~$g$ datum~$(P,R,\mathcal{L})$, we always denote the minimal even resolution of~$R$ by~$\widetilde{\psi}:\widetilde{P}\rightarrow P$ and the associated double cover datum by~$(\widetilde{R},\widetilde{\mathcal{L}})$. This datum defines a double cover~$\widetilde{\theta}:\widetilde{X}\rightarrow\widetilde{P}$ and hence by composition a not necessarily relatively minimal fibration~$\widetilde{f}:\widetilde{X}\rightarrow C.$ Its minimal model is then a relatively minimal, hyperelliptic fibration~$f:X\rightarrow C$ as shown in Figure~\ref{tildef}.
We say that the fibration~$f:X\rightarrow C$ is associated to the genus~$g$ datum.
\begin{Lemma}[{\cite[page~183]{compactcomplex}}]
\label{lemminv}
Let~$\widetilde{f}:\widetilde{X}\rightarrow C$ be a hyperelliptic genus~$g$ fibration defined by a genus~$g$ datum~$(P,R,\mathcal{L})$, where the branch divisor~$R$ has bidegree~$(2g+2,(g+1)e+n)$. Let~$f:X\rightarrow C$ be its minimal model and~$m$ be the number of vertical~$(-1)$-curves in~$\widetilde{X}$. Then we have
\item[i)] $\omega_{f}^2=c_1^2(X)-8(g-1)(g_C-1)=(2g-2)n-2\sum\limits_{i=1}^{r}\left(k_i-1\right)^2-m$,
\item[ii)] $\chi_{f}=\chi(\mathcal{O}_{X})-(g-1)(g_C-1)=\frac{1}{2}gn-\frac{1}{2}\sum\limits_{i=1}^r k_i\left(k_i-1\right)$,
where~$k_i:=\left\lfloor\frac{m_{p_i}}{2}\right\rfloor$.
\end{Lemma}
\begin{Rem}
It is worth mentioning that our invariants of the fibration~$f$ do not depend on~$e$ anymore, since we defined~$n$ in a way that~$e$ vanishes in any intersection. So, for example, we have
\begin{eqnarray*}
K_{P}^2=8(1-g_C), \quad R^2=2(2g+2)n & \text{ and } \quad K_{P}.R=(2g+2)(2g_C-2)-2n.
\end{eqnarray*} 
\end{Rem}
\begin{figure}[!htbp]
\begin{center}
	\begin{tikzpicture}[every node/.style={anchor=center},>=stealth']
	\matrix (m) [matrix of math nodes, row sep=3em, column sep=3em]
	{\widetilde{X} & {} & \widetilde{P}\supseteq\widetilde{R}\\
	 & & \hat{P}\supseteq\hat{R}\\
	X & & P\supseteq R \\
		 & C &  \\ };
	\draw[->] (m-1-1) -- node[left] {$\rho$} (m-3-1);
	\draw[->] (m-1-1) -- node[right] {$\widetilde{f}$} (m-4-2);
	\draw[->] (m-1-1) -- node[above] {$\hat{\theta}$} (m-1-3);
	\draw[->] (m-3-1) -- node[left] {$f$} (m-4-2);
	\draw[->] (m-3-3) -- node[right] {$\varphi$} (m-4-2);
	\draw[->] (m-1-3) -- node[right] {$\psi'$} (m-2-3);
	\draw[->] (m-2-3) -- node[right] {$\hat{\psi}$} (m-3-3);
	\end{tikzpicture}
\end{center}
\caption{\label{tildef}The hyperelliptic fibration associated to a genus~$g$ datum.}
\end{figure}
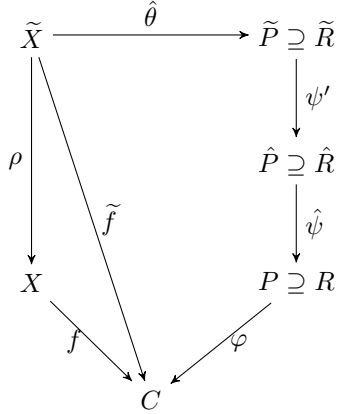

\subsubsection{Negligible singularities}
As just seen depend the invariants of~$f$ heavily on the multiplicity of singularities in~$R$. The idea of Xiao was now to cluster the singularities of~$R$ in two different types. Therefore, we consider all singularities that are living on the same exceptional components, i.e.\ if~$p_i\in P_{i-1}$ and~$p_j\in P_{j-1}$ for~$j<i$, such that~$\psi_j\circ...\circ\psi_{i-1}(p_i)=p_j$. We recall that this means that $p_i$ is infinitely near $p_j$.
\begin{Defi}
Let~$(P,R,\mathcal{L})$ be a genus~$g$ datum and let~$p\in R$ be a singularity of multiplicity~$m_p$. We call~$p$ a \emph{negligible singularity} if~$m_p\leq 3$ and for all~$p_i$ infinitely near~$p$ we have~$m_{p_i}\leq 3$. Consistently, we call a singularity that is not negligible a \emph{non-negligible singularity}.
\end{Defi}
Note that negligible singularities and simple singularities are the same. We can classify them by using so-called A-D-E types. The description negligible comes from the fact that a negligible singularity has no influence on the invariants of~$\widetilde{X}$ as we have just calculated in Lemma~\ref{lemminv}. 
\begin{Thm}[{\cite[Theorem II.8.1]{compactcomplex}}]
\label{ADE-sing}
Let~$p$ be a negligible singularity on the divisor~$R\subseteq P$. If~$m_p=2$ then the singularity is in~$R$ locally of the form
\begin{eqnarray*}
A_m: & y^2+z^{m+1}=0 & \text{where }m\geq 1.
\end{eqnarray*}
If~$m_p=3$ then the singularity is in~$R$ locally described by one of the following
\begin{eqnarray*}
D_m: & y^2z+z^{m-1}=0 & \text{where }m\geq 4,\\
E_6: & y^3+z^{4}=0, & \\
E_7: & y^3+yz^{3}=0, & \\
E_8: & y^3+z^{5}=0. & 
\end{eqnarray*}
\end{Thm}
Accordingly, we will speak of a \emph{(negligible) singularity of type~$A_m$, $D_m$, $E_6$, $E_7$ or~$E_8$}. Let us decompose the minimal even resolution~$\widetilde{\psi}$ into~$\widetilde{\psi}=\hat{\psi}\circ\psi'$, where~$\psi':\widetilde{P}\rightarrow\hat{P}$ is the blow-up of all negligible and~$\hat{\psi}:\hat{P}\rightarrow P$ the blow-up of all non-negligible singularities. We further denote by~$(\hat{R},\hat{\mathcal{L}})$ the reduced even image of~$(R,\mathcal{L})$ in~$\hat{P}$. Now the branch divisor~$\hat{R}$ is in a certain way 'smooth enough' to construct the double cover.
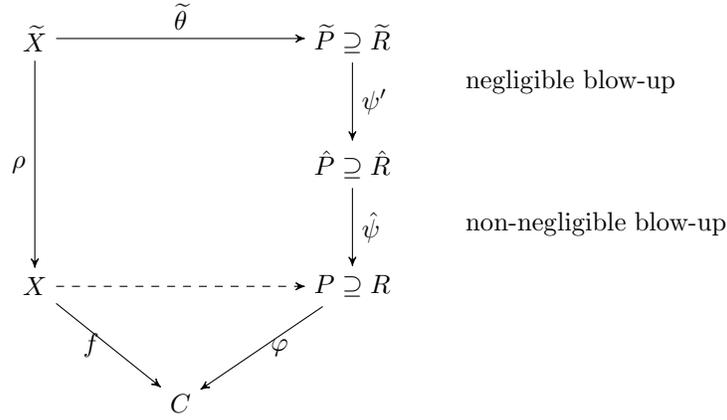
\begin{figure}[!htbp]
	\centering
	\begin{tikzpicture}[every node/.style={anchor=center},>=stealth']
	\matrix (m) [matrix of math nodes, row sep=3em, column sep=4em]
	{\widetilde{X} & {} & \widetilde{P}\supseteq\widetilde{R} & {} \\
	 & & \hat{P}\supseteq\hat{R} & {} \\
	X & & P\supseteq R & {} \\
		 & C &  {} \\ };
	\draw[->] (m-1-1) -- node[left] {$\rho$} (m-3-1);
	\draw[->] (m-1-1) -- node[above] {$\widetilde{\theta}$} (m-1-3);
	\draw[->] (m-3-1) -- node[left] {$f$} (m-4-2);
	\draw[->,dashed] (m-3-1) -- node[right] {} (m-3-3);
	\draw[->] (m-3-3) -- node[right] {$\varphi$} (m-4-2);
	\draw[->] (m-1-3) -- node[right] {$\psi'$} (m-2-3);
	\draw[->] (m-2-3) -- node[right] {$\hat{\psi}$} (m-3-3);
	\node[below right] at (2.5,2.1) {negligible blow-up};
	\node[below right] at (2.5,0.2) {non-negligible blow-up};
	\end{tikzpicture}
	\caption{\label{genusgdata2}Negligible and non-negligible blow-up.}
\end{figure}

\begin{Rem}
The pair~$(\hat{R},\hat{\mathcal{L}})$ is a double cover datum on~$\hat{P}$. It therefore defines a double cover~$\hat{\theta}:\hat{X}\rightarrow\hat{P}$, where the surface~$\hat{X}$ might be singular. Singularities arise over the singular points of~$\hat{R}$, i.e.\ over the negligible points. We assume that the double cover is locally given by
$$(x,y,z)\mapsto (y,z),$$
and the singularity lies in the origin. Then Theorem~\ref{ADE-sing} tells us that the surface singularities are locally described by
\begin{equation*}
\begin{aligned}
A_m: & \text{  } & x^2+y^2+z^{m+1} & =  0 &  \text{  } & \text{where }m\geq 1,\\
D_m: & \text{  } & x^2+y^2z+z^{m-1} & =  0 & \text{  } & \text{where }m\geq 4,\\
E_6: & \text{  } & x^2+y^3+z^{4} & =  0, & \\
E_7: & \text{  } & x^2+y^3+yz^{3} &  =  0, & \\
E_8: & \text{  } & x^2+y^3+z^{5} & =  0. & 
\end{aligned}
\end{equation*}
We call such a singularity a \emph{rational double point of type~$A_m$, $D_m$, $E_6$, $E_7$ or~$E_8$}. Note that any negligible singularity~$p$ satisfies~$2\lfloor\frac{m_p}{2}\rfloor=2$. Therefore the minimal even resolution~$\psi':\widetilde{P}\rightarrow\hat{P}$ of~$\hat{R}$ describes now a unique resolution~$\sigma:X\rightarrow\hat{X}$ of the rational double points in~$\hat{X}$ as shown in Figure~\ref{resdoublepoint}.
\begin{figure}[!htbp]
	\centering
	\begin{tikzpicture}[every node/.style={anchor=center},>=stealth']
	\matrix (m) [matrix of math nodes, row sep=3em, column sep=4em]
	{\widetilde{X} & \widetilde{P}\supseteq\widetilde{R}  \\
	 \hat{X} & \hat{P}\supseteq\hat{R} \\};
	\draw[->] (m-1-1) -- node[left] {$\rho'$} (m-2-1);
	\draw[->] (m-1-1) -- node[above] {$\widetilde{\theta}$} (m-1-2);
	\draw[->] (m-1-2) -- node[left] {$\psi'$} (m-2-2);
	\draw[->] (m-2-1) -- node[above] {$\hat{\theta}$} (m-2-2);
	\end{tikzpicture}
	\caption{\label{resdoublepoint}The resolution of the rational double points.}
\end{figure}
\end{Rem}
\begin{Prop}[\cite{compactcomplex}]
\label{propdoublepoints}
The exceptional curve of the minimal resolution of a rational double point of type~$A_m$, $D_m$, $E_6$, $E_7$ and~$E_8$ consists of a configuration of $(-2)$-curves. The index of the type equals the number of~$(-2)$-curves in the resolution.
\end{Prop}
The fibrations~$f:X\rightarrow C$ associated to a genus~$g$ datum is not necessarily semi-stable but we can use a sufficient criterion.
\begin{Prop}
\label{propsemistableAM}
Let~$f:X\rightarrow C$ be a semi-stable, hyperelliptic fibration and let~$(P,R,\mathcal{L})$ be the corresponding genus~$g$ datum. Then~$\hat{R}$ contains as negligible singularities only singularities of type~$A_m$ for~$m\geq 1$ and all ramification points in~$\hat{R}$ of the morphism~$\hat{R}\rightarrow C$ are simple ramification points.
\end{Prop}
\begin{proof}
It is clear that all ramification points in~$\hat{R}$ must be simple. So let us assume that~$\hat{R}$ contains a negligible singularity. The resolution graph of any negligible singularity other than of type~$A_m$ contains a $(-2)$-curve in~$X$ that does intersect the other components in exactly one point (see Remark~\ref{ADE-sing}). In this case the corresponding fiber of~$f$ is not semi-stable.
\end{proof}

\begin{Ex}[{\cite[Example 5.2]{beauminimum}}] 
\label{bspbeauville}
	Let us construct an example of a semi-stable hyperelliptic fibration attaining equality in the lower slope bound. Let~$\varphi:C\rightarrow\P^1$ be a cover of degree~$n$, ramified over the branch points~$R\subseteq\P^1$ and let~$u$ be an automorphism of~$\P^1$ such that~$R$ contains no fixed point of~$u$ but satisfies~$u(R)\subseteq R$. We assume further that all ramification points of~$\varphi$ are of index~$2$. Then in~$C\times\P^1$ we regard the two divisors~$\Gamma_{\varphi}$ and~$\Gamma_{u\circ\varphi}$, defined by the two graphs of~$\varphi$ and~$u\circ\varphi$. Since the automorphism group of~$\P^1$ is~$\mathrm{PGL}(2,\mathbb{C})$, the divisors are linear equivalent, i.e.~$\mathcal{O}_{C\times\P^1}(\Gamma_{\varphi}\cup\Gamma_{u\circ\varphi})=\mathcal{L}^2$ for some~$\mathcal{L}\in\Pic(C\times\P^1)$. We can therefore construct a double cover $X'\rightarrow C\times\P^1$, ramified over~$\Gamma_{\varphi}\cup\Gamma_{u\circ\varphi}$. We denote with~$f':X'\rightarrow\P^1$ the composition of the double cover and the second projection. For~$t\in\PP^1$ the fiber~$(f')^{-1}(t)$ is a double cover of the curve~$C$, ramified along the divisor~$\varphi^{-1}(t)+\varphi^{-1}(u^{-1}(t))$. Since this divisor is almost everywhere reduced, except for the points in~$R$ and the two fixed points of~$u$, where it has multiplicity~$2$, the fiber has only rational singularities, i.e.\ it is a stable curve. By the Riemann-Hurwitz formula we calculate the fiber genus as~$g=2g_C+n-1$. Let~$X\rightarrow X'$ be the canonical resolution of the singularities. These singularities appear only over~$\Gamma_{\varphi}\cap\Gamma_{u\circ\varphi}$, so if we consider the now semi-stable fibration~$f:X\rightarrow\P^1$, we have~$|R|+2$ singular fibers.
\\
Let us restrict to the case~$C=\P^1$. The graph~$\Gamma_{\varphi}$ is a divisor of type~$(n,1)$ and therefore we can calculate the self-intersection~$\Gamma_{\varphi}^2=2n$. Since the canonical divisor is $K_{\P^1\times\P^1}=p_1^*K_{\P^1}+p_2^*K_{\P^1}$ we get~$\Gamma_{\varphi}.K_{\P^1\times\P^1}=-2n-2$. By applying Lemma~\ref{lemminv} we finally find the invariants of the fibration:
	\begin{equation*}
	\begin{aligned}
	\chi_f=1+(g-1), &\text{  } & \omega_f^2=8-4n+8(g-1), &\text{ and } & \delta_f=4+4n+4(g-1).
	\end{aligned}
	\end{equation*}
	Using the morphism~$\varphi$ of degree~$4$, defined by~$\varphi(t)=t^2+1/t^2$ (as a quotient of~$\P^1$ by~$\Z/2\Z\times\Z/2\Z$), we get a genus~$3$ fibration with~$|R|=3$ and we conclude that its speed is~$L_f=2$ and its slope~$\lambda_f=8/3$.
\end{Ex}

\begin{Rem}
In~\cite{beauszpiro}, Beauville referred to~{\cite[Example~5.2]{beauminimum}} as an example of a semi-stable genus~$3$ fibration with five singular fibers such that~$\delta_f=40$ without giving any further details. If this were true then it would yield an example of high speed, contradicting Noether's formula. However, the number of double points of this fibration is~$\delta_f=28$.
\end{Rem}

\subsection{Hyperelliptic fibrations with high speed}
\label{hypexamples}
Using a genus~$g$ datum (see Definition~\ref{defgenusgdatum}), i.e.\ a branch divisor on a ruled surface, we can construct interesting examples of hyperelliptic fibrations. As the invariants of the fibration, namely~$\omega_f^2$ and~$\chi_f$, depend only on the bidegree of the branch divisor and its type of singularities by Lemma~\ref{lemminv}, we can directly compute slope and speed and find examples with high speed. All examples here are constructed on the ruled surface~$P:=\P^1\times C$, which has invariant~$e=0$. The main theorem of this section is then the following.

\thmexistenz*
\begin{proof}
We will construct the particular examples for odd genus in Example~\ref{exampleoddgenus}, for even genus~$g\geq 4$ in Example~\ref{exampleevengenus}, for genus~$2$ in Example~\ref{examplegenus2}, for genus~$3$ in Example~\ref{examplegenus3} and for~$v)$ in Example~\ref{examplegenus8}.
\end{proof}

The existence of a branched cover~$C\rightarrow\widetilde{C}$ between curves with given branching data is a non-trivial question, known as the Hurwitz existence problem. We follow here~\cite{hurwitzproblem}, which provides a nice overview over this problem. As we are working over the base field~$\CC$ all curves are in fact Riemann surfaces, hence in particular orientable.
\begin{Defi}
Let~$C$ and~$\widetilde{C}$ be smooth curves, let~$m\geq 0$ and~$d\geq 2$ be integers and let~$(d_{ij})_{j=1,\dots,m_i}$ be a partition of~$d$. We call the $5$-tuple $(C,\widetilde{C},m,d,(d_{ij}))$, where~$i=1,\dots,m$, the \emph{branch datum} of a candidate branched cover. We further associate the number~$\widetilde{m}=m_1+\dots+m_n$ to such a datum.
\end{Defi}
We say that a branch datum is \emph{compatible}, if the following two conditions hold
\begin{itemize}
\item[i)] $2-2g_C-\widetilde{m}=d\cdot(2-2g_{\widetilde{C}}-m)$,
\item[ii)] $m\cdot d-\widetilde{m}$ is even.
\end{itemize}
It is easy to see that any branched cover defines a compatible branch datum, but it is still working process in which cases a compatible branch datum can actually realize a branched cover, i.e.\ there is a branched cover~$C\rightarrow\widetilde{C}$ of degree~$d$ branched over~$m$ points of~$\widetilde{C}$, such that the preimage of the $i$-th branch point consists of~$m_i$ points with local degree~$(d_{ij})$. We use here the following theorem.
\begin{Prop}[{\cite[Proposition~5.2]{zheng}}]
\label{existencebranchcover}
Let~$(C,\widetilde{C},m,d,(d_{ij}))$ be a compatible branch datum such that~$\widetilde{C}=\P^1$ and one of the partitions~$(d_{ij})$ is given by~$(d)$ only. Then this datum realizes a branched cover.
\end{Prop}
By Proposition~\ref{existencebranchcover} it is clear that~$\P^1\cong\CC\cup\{\infty\}$ has a ramified cover of degree~$n>0$ over itself, which is totally ramified at exactly two points. Let~$g\geq 2$ be a number, then there exists a double cover~$a_1:\P^1\rightarrow\P^1$ which is ramified at~$1$ and~$\infty$ and~$a_2:\P^1\rightarrow\P^1$ a cover of degree~$g+1$, totally ramified at the two points of~$a_1^{-1}(0)$. We take the composition~$a:=a_2\circ b_1:\P^1\rightarrow\P^1$, which is a~$(2g+2:1)$-cover over the projective line.
\subsubsection{Odd genus} 
\begin{Ex}
\label{exampleoddgenus}
Let~$g\geq 3$ be an odd number and let
$$(C,\P^1,3,4,((4),(4),(2,2)))$$
be a compatible branch datum. By Proposition~\ref{existencebranchcover}, this branch datum is realizable, we denote it by~$b:C\rightarrow\P^1$. We take the branch points at~$0$, $1$ and~$\infty$ in~$\P^1$, whose preimages have local degree~$((4),(4),(2,2))$. We now consider the divisor~$R$ that is described by~$a$ and~$b$ in~$\P^1\times C$, namely
$$R=\{(y,z)\in\P^1\times C|a(y)=b(z)\}\subseteq\P^1\times C$$
and we look at the projection on the second component, i.e.\ on~$C$.
\begin{figure}[!htbp]
	\centering
	\begin{tikzpicture}[every node/.style={anchor=center},>=stealth']
	\matrix (m) [matrix of math nodes, row sep=3em, column sep=3em]
	{R & \\
	 \P^1\times C & C \\
		  \P^1 & \P^1 \\ };
	\draw[->] (m-1-1) -- node[left] {$\subseteq$} (m-2-1);
	\draw[->] (m-2-1) -- node[below] {$pr_2$} (m-2-2);
	\draw[->] (m-2-1) -- node[left] {$pr_1$} (m-3-1);
	\draw[->] (m-3-1) -- node[above] {$a$} (m-3-2);
	\draw[->] (m-2-2) -- node[right] {$b$} (m-3-2);
	\draw[->] (m-1-1) -- node[right] {$pr_2\mid_{R}$} (m-2-2);
	\end{tikzpicture}
	\caption{\label{figbranchdivisor}}
\end{figure}
\\Our divisor~$R$ has a bidegree of~$(2g+2,4)$ on the ruled surface~$pr_2:\P^1\times C\rightarrow C$ (see Figure~\ref{figbranchdivisor}). The map~$pr_2|_{R}:R\rightarrow\P^1$ can be seen as the pullback of~$a$ under the base change~$b$ and therefore the set of critical points on the projective line is exactly~$\Delta:=b^{-1}(\{0,1,\infty\})$. Moreover we then have~$s:=|\Delta|=4$.
\\
Locally~$R$ is defined as~$a(y)=b(z)$ and so over~$b^{-1}(1)$ the divisor~$R$ has~$g+1$ singular points defined by~$y^2-z^4=0$, whereas over the two points of~$b^{-1}(\infty)$ there are~$g+1$ points defined by~$y^2-z^2=0$. This is precisely the definition of a singularity of type~$A_{3}$, respectively of type~$A_{1}$ (see Theorem~\ref{ADE-sing}) and in particular these singularities are negligible. The remaining two singularities are living over~$b^{-1}(0)$ and are defined by~$y^{g+1}-z^4=0$. Hence the even divisor~$R$ defines a genus~$g$ datum and therefore a fibration~$f:X\rightarrow C$ of genus~$g$. The set of critical points, whose fibers are singular, is given by~$\Delta$. Let us recall Lemma~\ref{lemminv} which says that
$$\chi_{f}=\chi(\mathcal{O}_{X})-(g-1)(g_C-1)=\frac{1}{2}gn-\frac{1}{2}\sum\limits_{i=1}^r k_i\left(k_i-1\right),$$
where~$k_i:=\left\lfloor\frac{m_{p_i}}{2}\right\rfloor$. Hence all we need to know to calculate the slope of this fibration is the number of non-negligible singularities and their multiplicity. Let us focus on the two singular points over~$b^{-1}(0)$, locally defined by~$y^{g+1}-z^4=0$, hence of multiplicity four. We need to apply an even blow-up~$\psi_1:P_1\rightarrow P$ to these two singularities. We consider the blow-up locally as~$\widetilde{\CC^2}\rightarrow\CC^2$ centered at~$p_1=(0,0)$ by
$$\widetilde{\CC^2}=\{(y,z),[u:v]\in\CC^2\times\P^1\text{ }|\text{ }yv=zu\}.$$
In the chart~$u=1$ we get 
$$y^{g+1}-y^4v^4=y^4(y^{g-3}-v^{4})=0.$$
The component~$z^4=0$ is corresponding to~$2\lfloor\frac{m_{p_1}}{2}\rfloor E_1=4E_1$, where~$E_1$ is the exceptional divisor of the blow-up. The even inverse image~$R_1$ is by definition described locally by~$y^{g-3}-v^{4}=0$. If~$g=3$, then the even inverse image is smooth. If~$g=5$, then on the intersection point of~$R_1$ and~$E_1$, namely on the point~$p_2=(0,0),[0,1]$, we have a unique singularity of multiplicity two, more concretely a singularity of type~$A_3$. If~$g\geq 7$, we blow up further the singularity of multiplicity four on~$R_1$ by~$\psi_2:P_1\rightarrow P_2$, where we assume that~$R_1$ is given by~$y^{g-3}-v^{4}=0$. We continue this process for both singular points of multiplicity~$4$ on~$R$ until the reduced inverse even image is smooth or all singular points living on the reduced inverse even image are negligible. Let~$s_4(f)$ be the number of all such singularities of multiplicity four, including infinitely near ones. Then by just counting we get
$$s_4(f)=2\cdot\left(\left\lfloor\frac{g+1}{4}\right\rfloor\right).$$
The resulting fibration is semi-stable by Proposition~\ref{propsemistableAM}. Now we get~$\chi_f=2g-2k$, where~$k:=\left\lfloor\frac{g+1}{4}\right\rfloor$. Note in addition that by the properties of a compatible branch datum we get that~$2g_C-2+s=4$. Hence we calculate as speed~$L_f=g-k$.
\end{Ex}

\begin{Ex}
\label{examplegenus3}
Let~$g=3$ and consider the compatible branch datum
$$(C,\P^1,3,3,((3),(3),(3))).$$
Here~$C$ is a smooth curve of genus~$g_C=2$ and we denote by~$b:C\rightarrow\P^1$ the corresponding cover, branched over~$0$, $1$ and~$\infty$. We again take the divisor
$$R':=\{(y,z)\in\P^1\times C|a(y)=b(z)\}\subseteq\P^1\times C$$
in the ruled surface defined by the projection~$pr_2:\P^1\times C\rightarrow C$. Let~$F_0$ be the fiber of~$pr_2$ over the point~$b^{-1}(0)$, then the divisor~$R:=R'+F_0$ is of bidegree~$(8,4)$ and induces a semi-stable fibration~$f:X\rightarrow C$ of genus~$3$. Over~$b^{-1}(0)$ there are two singular points of multiplicity four, whose infinitely near points are smooth. Hence we have
$$L_f=\frac{8}{3},$$
since~$2g_C-2+s=3$.
\end{Ex}

\begin{Ex}
Let~$g\geq 2$ be a number such that~$g\equiv 1\mod 4$. Let~$b:\P^1\rightarrow\P^1$ be a cover of degree~$4$, ramified over the two points~$0$ and~$\infty$. The divisor
$$R=\{(y,z)\in\P^1\times\P^1|a(y)=b(z)\}\subseteq\P^1\times\P^1$$
is an even divisor of bidegree~$(2g+2,4)$ of the ruled surface~$pr_1:\P^1\times\P^1\rightarrow\P^1$ and defines therefore a relatively minimal fibration~$\widetilde{f}:X\rightarrow\P^1$ of genus~$g$. The number of singular fibers is here~$s=1+1+4=6$. There are only two non-negligible singular points over~$b^{-1}(0)$, locally defined by~$y^{g+1}-z^4=0$. By blowing up we calculate
$$s_4(f)=2\cdot\left\lfloor\frac{g}{4}\right\rfloor$$
which implies
$$\chi_f=2g-2\cdot\left\lfloor\frac{g}{4}\right\rfloor.$$
Note that the fibration is again semi-stable, hence we calculate as speed
$$L_f=g-\left\lfloor\frac{g}{4}\right\rfloor.$$
\end{Ex}
\begin{Ex}
Let~$g\geq 7$ be a number such that~$g\equiv 1\mod 6$. We consider by~$b:\P^1\rightarrow\P^1$ a cover of degree~$6$, totally ramified at~$0$ and~$\infty$.
The divisor
$$R=\{(y,z)\in\P^1\times\P^1|a(y)=b(z)\}\subseteq\P^1\times\P^1$$
is an even divisor of bidegree~$(2g+2,6)$ of the ruled surface~$pr_1:\P^1\times\P^1\rightarrow\P^1$ and defines therefore a relatively minimal fibration~$f
:\widetilde{X}\rightarrow\P^1$ of genus~$g$. The number of singular fibers is here~$s=1+1+6=8$. There are only two non-negligible singular points over the two points of~$b^{-1}(0)$, locally defined by~$y^{g+1}-z^6=0$. As before we let~$s_6(f)$ be the number of all such singularities of multiplicity six, including infinitely near ones. By blowing up we calculate
$$s_6(f)=2\cdot\left\lfloor\frac{g}{6}\right\rfloor$$
which implies
$$\chi_f=3g-6\cdot\left\lfloor\frac{g}{6}\right\rfloor.$$
Note that the fibration is again semi-stable, hence we calculate as speed
$$L_f=g-2\cdot\left\lfloor\frac{g}{6}\right\rfloor.$$
\end{Ex}

\subsubsection{Even genus}
\begin{Ex}
\label{exampleevengenus}
To construct a fibration of even genus~$g\geq 4$, we take the compatible branch datum
$$(C,\P^1,3,2g+2,((g+1,g+1),(2g+2),(2g+2)))$$
which realizes a branched cover~$b:C\rightarrow\P^1$ by Proposition~\ref{existencebranchcover}. Let~$0$, $1$ and~$\infty$ in~$\P^1$ be the branch points whose preimages have as local degree~$((g+1,g+1),(2g+2),(2g+2))$. Now the branch divisor~$R$ is given by
$$R=\{(y,z)\in\P^1\times C|a(y)=b(z)\}\subseteq\P^1\times C$$
where again we consider the projection to the second component. Therefore~$R$ has a bidegree of~$(2g+2,2g+2)$ on the ruled surface~$pr_2:\P^1\times C\rightarrow C$ and the set of critical points is~$\Delta:=b^{-1}(\{0,1,\infty\})$, with~$s:=|\Delta|=4$.
\\
Again~$R$ is locally defined as~$a(y)=b(z)$ and so over~$b^{-1}(1)$ and~$b^{-1}(\infty)$ the divisor~$R$ has each~$g+1$ singular points defined by~$y^2-z^{2g+2}=0$. This is precisely the definition of a singularity of type~$A_{2g+1}$. The remaining two singularities are living over the two points of~$b^{-1}(0)$ and are defined by~$y^{g+1}-z^{g+1}=0$. Again~$R$ is even and defines a genus~$g$ datum, hence a relatively minimal fibration~$f:X\rightarrow C$ of genus~$g$. The set of critical points, whose fibers are singular, is given by~$\Delta$. 
After applying an even blow-up on the singularity~$y^{g+1}-z^{g+1}=0$, the even inverse image~$R_1$ is smooth. Hence the singularity is of odd multiplicity~$g+1$ and neither the second component of a singularity of type~$(g-1\rightarrow g-1)$ nor the first component of a singularity of type~$(g+1\rightarrow g+1)$. We say that this tuple of singular points is a singularity is of type~$s_g$ and in particular we get~$s_g(f)=4$. Now by Lemma~\ref{lemminv} we get that
$$\omega_f^2=2g^2+8g-12 \text{   and   } \chi_f=\frac{g^2}{2}+2g.$$
By the same argument as before~$2g_C-2+s=2g+2$ and so in terms of slope and speed we calculate
$$\lambda_f=4-\frac{24}{g^2+4g} \text{   and   } L_f=g-\frac{g^2-2g}{2g+2}.$$
\end{Ex}

\begin{Ex}
\label{examplegenus2}
Let~$g=2$ and likewise to Example~\ref{examplegenus2} we consider the branch datum
$$(C,\P^1,3,5,((5),(5),(5)))$$
with the associated branched cover~$b:C\rightarrow\P^1$ and define
$$R':=\{(y,z)\in\P^1\times C|a(y)=b(z)\}\subseteq\P^1\times C$$
as divisor on the ruled surface~$pr_2:\P^1\times C\rightarrow C$. Let~$F_0$ be the fiber of~$pr_2$ over the point~$b^{-1}(0)$. Then the divisor~$R:=R'+F_0$ is of bidegree~$(6,6)$ and defines a semi-stable fibration~$X\rightarrow C$. The divisor has two singular points of multiplicity three over~$b^{-1}(0)$. By applying the local even blow-up, the even inverse image~$R_1$ consists of two components, the first is one copy of the exceptional curve and the latter the strict transform of~$R$. On their intersection point, namely on the point we have a unique singularity of multiplicity four. This allows us to conclude~$\chi_f=4$ and therefore as speed
$$L_f=2\cdot\frac{\chi_f}{2g_C-2+s}=2\cdot\frac{4}{5}=\frac{8}{5}.$$
\end{Ex}

\begin{Ex}
\label{examplegenus8}
Let~$g\equiv 0\mod 4$ and let
$$(C,\P^1,3,4,((4),(4),(2,2)))$$
be the same compatible branch datum as in Example~\ref{exampleoddgenus} defining the same cover~$b:C\rightarrow\P^1$. We also take as divisor in~$P:=\P^1\times C$ here
$$R=\{(y,z)\in\P^1\times C|a(y)=b(z)\}\subseteq\P^1\times C.$$
Blowing evenly up $k$-times the two singularities defined by~$y^{g+1}-z^4=0$ resolves this singularities, where
$$k:=\left\lfloor\frac{g+1}{4}\right\rfloor.$$
Hence by the same calculation as before we get~$s_4(f)=2\cdot(\lfloor\frac{g+1}{2}\rfloor)$ and therefore
$$\chi_f=2g-2k\text{ as well as } L_f=g-k.$$
\end{Ex}

\printbibliography

\end{document}